\documentclass[11pt]{amsart}
\usepackage{graphicx, amsmath, fullpage, amssymb, amsthm, amsfonts, mathrsfs, eucal, moreverb, mathtools, float, bm}

\usepackage[p,osf]{cochineal}
\usepackage[scale=.95,type1]{cabin}
\usepackage[zerostyle=c,scaled=.94]{newtxtt}

\usepackage{enumerate}
\usepackage[left=2 cm,right=2cm,top=2cm,bottom=2cm]{geometry}

\newtheorem{theorem}{Theorem}[section]
\newtheorem{lemma}[theorem]{Lemma}
\newtheorem{prop}[theorem]{Proposition}

\theoremstyle{definition}
\newtheorem{definition}[theorem]{Definition}

\usepackage[toc,page]{appendix} 
\usepackage[svgnames]{xcolor}
\theoremstyle{remark}
\newtheorem{remark}[theorem]{Remark}
\numberwithin{equation}{section}
\usepackage[colorlinks,
            linkcolor=FireBrick,
            anchorcolor=red,
            citecolor=RoyalBlue
            ]{hyperref}
            
\usepackage{todonotes}

\DeclareMathSymbol{e}{\mathalpha}{operators}{`e}
 \DeclareMathSymbol{I}{\mathalpha}{operators}{`I}
 \DeclareMathSymbol{D}{\mathalpha}{operators}{`D}
\DeclareMathSymbol{J}{\mathalpha}{operators}{`J}
\DeclareMathSymbol{K}{\mathalpha}{operators}{`K}
\DeclareMathSymbol{L}{\mathalpha}{operators}{`L}
\DeclareMathSymbol{H}{\mathalpha}{operators}{`H}

\newcommand{\tr}{\mathrm{tr}}

\newcommand{\E}{\mathbb E}
\renewcommand{\P}{\mathbb P}
\newcommand{\N}{\mathbb N}
\newcommand{\C}{\mathbb C}
\newcommand{\1}{\mathbf 1}

\title{The characteristic polynomial of sums of random permutations \\ and regular digraphs}

\author[A1]{Simon Coste}
\address{D\'epartement d'informatique, ENS - PSL University}
\email{simon.coste@ens.fr}

\author[A2]{Gaultier Lambert}
\address{Institute of Mathematics, University of Zurich}
\email{glambert@kth.se}

\author[A3]{Yizhe Zhu}
\address{Department of Mathematics, University of California, Irvine}
\email{yizhe.zhu@uci.edu}

\date{\today}

\begin{document}

\maketitle

\begin{abstract}
Let $A_n$ be the sum of $d$ permutation matrices of size $n\times n$, each drawn uniformly at random and independently. We prove that the normalized characteristic polynomial  $\frac{1}{\sqrt{d}}\det(I_n - z A_n/\sqrt{d})$ converges when $n\to \infty$ towards a random analytic function on the unit disk. As an application, we obtain an elementary proof of the spectral gap of random regular digraphs. Our results are valid both in the regime where $d$ is fixed and for $d$ slowly growing with $n$. 
\end{abstract}

\section{Introduction}

Spectral properties of non-Hermitian random matrices can have different behaviors depending on their degree of sparsity. These properties are now well understood for dense matrices with iid entries; a well-known example is the Circular Law \cite{bordenave2012around}, for which the optimal sparsity threshold is known \cite{wood2012universality,basak2019circular,rudelson2019sparse}. However, when the matrices in question are very sparse, with a fixed number of non-zero entries on each row, including dependencies, the problem becomes different and more challenging. One of our goals in this paper is to understand these differences.

\subsubsection*{Sums-of-permutations} There are numerous ways to enforce sparsity in random matrices, and different ensembles are expected to behave differently. In this paper, we focus on \emph{permutation matrices}, that is, matrices with exactly one non-zero entry on each row and each column; adding $d$ independent permutation matrices chosen uniformly at random, we obtain a random matrix $A$ with integer entries, whose row/column-sums  are all exactly equal to~$d$. Thus, $A$ can be viewed as a typical matrix with fixed row sums and column sums. This model displays two important properties: 
\begin{enumerate}[(i)]
  \item The structure of $A$ is very constrained (fixed row/column sums), hence the entries of $A$ are not independent. 
  \item The rows and columns of $A$ can be swapped while keeping the distribution of $A$ invariant (invariance by permutations). 
\end{enumerate}

\subsubsection*{Random regular directed graphs} Sums of random permutations are  of particular interest to random graph theory since they are a popular proxy for the (adjacency) matrix of \emph{random regular digraphs}. A digraph is \textit{regular} when each node has the same number of in-neighbors and out-neighbors; consequently, the matrix $A$ defined above is the matrix of a $d$-regular digraph, possibly with multiple edges. It turns out that for fixed $d$, conditioning on $A$ having no entry greater than $1$, the graph represented by $A$ is nearly uniform among all the $d$-regular directed graphs (the two models are \emph{contiguous} \cite{janson1995random}, and both are contiguous with the configuration model). Our analysis thus provides results on the eigenvalues of random regular digraphs, notably a new proof for the directed version of Friedman's second eigenvalue theorem \cite{friedman2008proof,bordenave2020new} which relates to important graph-theoretical notions such as graph expansion and random walks on digraphs \cite{parzanchevski2020ramanujan, coste2021spectral}.

\subsubsection*{Spectral properties} In general, studying the eigendecomposition of non-Hermitian matrices can be challenging.  
For an $n\times n$ permutation matrix ($d=1$),  the eigenvalues of $A$ belong to the $n$-th roots of unity, and their multiplicity is given by the cycle decomposition theorem; if $c_k$ denotes the number of cycles of length $k$ (thus $c_1 + 2c_2 + \dotsb + n c_n = n$ ), then the multiplicity of a root of unity $\omega$ will be the sum of the $c_k$ for which $\omega^k=1$.  
This basic structure allows to study even non-uniform  permutation matrices \cite{hughes2013random,ben2015fluctuations}.
However, already for $d=2$, this straightforward analysis breaks down because generically,  the permutation matrices that we sum do not commute. In particular, a famous conjecture states that for fixed $d$ as $n\to\infty$, the empirical spectral measure of $A$ converges to the \textit{oriented Kesten-McKay} distribution \cite{bordenave2012around}. 
In contrast, in the regime where $d\to\infty$ as $n\to\infty$, one expects to recover the circular law and there are already several results in this direction \cite{basak2018circular,cook2019circular,litvak2020circular}.

\subsubsection*{Our contributions}In this paper, we study asymptotics of the characteristic polynomial outside of the spectral support. 
 The results we obtain are analogous to \cite{bordenave2020convergence,najnudel2020secular}; we identify the distributional functional limit of the characteristic polynomial of $A$, away from the spectrum.
This problem has also been considered in the Hermitian case for Gaussian $\beta$-ensembles \cite{LP}. 
In particular, the relationships between our results and the theory of multiplicative chaos are discussed in Section~\ref{sec:lcf}.
Surprisingly, we obtain the same limiting functions as in \cite{coste2021sparse}, which considers non-Hermitian matrices with independent $\mathrm{Bernoulli}(d/n)$ entries --- a proxy for sparse Erd\H{o}s-R\'enyi digraphs. This result is rather unexpected since in the Hermitian case, the spectral properties of random regular graphs and sparse Erd\H{o}s-R\'enyi graphs are radically different for fixed $d$. In Appendix \ref{app:ewens}, we also report on several observations regarding the sum of Ewens-distributed random permutations, which can be of independent interest.

\bigskip

As corollaries of our convergence results for characteristic polynomials (Theorems \ref{thm:mc} and \ref{thm:gmc} stated in Section \ref{sec:results}), we obtain  two spectral gap theorems covering different regimes.   {We say a digraph is \emph{simple} if its directed adjacency matrix has no entry greater than 1}.

\begin{theorem}[spectral gap for $d$-regular digraphs, $d$ fixed]\label{cor:spectral_gap_fixed_d}
  Let $A_n$ be one of the following random matrix models: 
  \begin{enumerate}[(i)]
  	\item the sum of $d$ independent uniform permutation matrices of size $n\times n$;
  	\item the sum of $d$ independent $n\times n$ uniform permutation matrices conditioned on their graph being simple;
  	\item the adjacency matrix of a uniform random directed $d$-regular graph on $n$ vertices.
  \end{enumerate} Then, for any $\varepsilon>0$ and $d\geq 1$, 
\[
   \lim_{n\to\infty} \mathbb{P}(|\lambda_2| > \sqrt{d} + \varepsilon)= 0.
\]
  where $\lambda_i$ are the complex eigenvalues of $A_n$, ordered by decreasing modulus: $ d = \lambda_1 \geqslant |\lambda_2| \geqslant \dotsb \geqslant |\lambda_n|$.  
\end{theorem}
Theorem \ref{cor:spectral_gap_fixed_d} gives an alternate proof of the spectral gap result obtained in \cite{coste2021spectral}; it notably implies that there are no outliers outside the support of the oriented Kesten-McKay law, except the trivial eigenvalue $\lambda_1=d$. It is not known for the moment how to prove the matching lower-bound on $|\lambda_2|$; regarding this problem (directed analogs of the Alon-Boppana inequality), we refer to the discussion in \cite[Section 4]{coste2021sparse}.

When $d(n)\to\infty$, there is no contiguity result in the literature between the uniform regular digraphs and sums of random permutations (it was conjectured in \cite{cook2017discrepancy} these two models are contiguous when $d(n)=O(\log n)$); therefore, our next Theorem only applies to the sums-of-permutations model.

\begin{theorem}[spectral gap for sums of permutations, $d(n) \to \infty$]\label{cor:spectral_gap_growing_d}
  Let $A_n$ be the sum of $d(n)$ independent $n\times n$ uniform random permutation matrices. 
  Assume that $d(n)\to\infty$ in such a way that $d(n)= n^{o(1)}$ as $n\to\infty$. 
  Then, for any $\varepsilon>0$, 
  \begin{eqnarray}\label{eq:spectal_gap_limit}
    \mathbb{P}(|\lambda_2| > \sqrt{d(n)}(1 + \varepsilon)) \to 0
  \end{eqnarray}
  where $\lambda_i$ are the complex eigenvalues of $A$, ordered by decreasing modulus: $ d = \lambda_1 \geqslant |\lambda_2| \geqslant \dotsb \geqslant |\lambda_n|$. 
\end{theorem}

Combining the circular law proved in \cite{basak2018circular} for the random matrix  $A_n/\sqrt{d(n)}$, the bound \eqref{eq:spectal_gap_limit} is sharp (at least in the regime $d(n)\geq \log^{12}(n)/(\log\log n)^4$ and $d(n)= n^{o(1)}$). With high probability, there is no outlier eigenvalues besides the trivial one $\lambda_1=\sqrt{d(n)}$.	
In Theorem \ref{cor:spectral_gap_growing_d}, the condition $d(n)= n^{o(1)}$ is technical (due to our proof of Theorem \ref{thm:trace} (2)) and it could be dispensed with some extra work.  {For denser regimes, like $d\asymp n$, other techniques apply and the spectral behaviour of $A$ is expected to be close to the Ginibre ensemble \cite{basak2018circular}.}

\subsection*{Related work}
The spectral properties of random permutations under Ewens distribution were investigated in several works, including the characteristic polynomial  and linear statistics \cite{hughes2013random,dang2014characteristic,ben2015fluctuations,bahier2019number,bahier2019characteristic,bahier2021smooth}. For a uniform random permutation matrix, the maximum of the characteristic polynomial on the unit circle has been studied in \cite{cook2020maximum}. 
It is also of interest to study this question for sums-of-permutations for general $d$ and we intend to consider this problem in subsequent work.

Spectral gap of random $d$-regular digraphs with fixed $d$ was studied in \cite{coste2021spectral} using the high trace method, which is limited to fixed $d$. 
 For uniform random regular digraphs, it was shown in \cite{zhu2020second} that $|\lambda_2|\leq C\sqrt{d(n)}$ for $d(n)=O(n^{2/3})$, and for $n^{\varepsilon}\leq d(n)\leq n/2$ in \cite{tikhomirov2019spectral}.  By the size-biased coupling method introduced in \cite{cook2018size}, one can prove a similar result to \cite[Theorem 2.6]{cook2018size} that $|\lambda_2|=O(\sqrt{d(n)})$ for the sum of permutations for any $2\leq d(n)\leq n/2$. However, all of these results rely on the $\varepsilon$-net and the Kahn-Szemr\'{e}di argument \cite{friedman1989second}, which yields an absolute constant far from optimal. Our approach gives a unified treatment for fixed and growing $d$ in the random permutation model with a sharp constant,  while being elementary on a technical level.

Lower bounds on the least singular value of random $d$-regular digraphs for fixed $d$ is a problem of considerable interests since it is a crucial step towards proving the conjectured oriented Kesten-McKay law.  The singularity probability was estimated in \cite{huang2021invertibility,meszaros2020distribution} for fixed $d$, and in \cite{cook2017singularity,litvak2017adjacency} for growing $d(n)$. Quantitative estimates on the least singular values for growing $d(n)$ were obtained in \cite{basak2018circular,cook2019circular,litvak2019smallest,jainsmallest}.

In the Hermitian case, the sum of random permutations and their transpose  $\sum_{i=1}^d (P_i+P_i^\top)$ can be seen as a model for random $2d$-regular multi-graphs, and their spectral properties have been investigated, including spectral gaps \cite{friedman2008proof,cook2018size} and linear eigenvalue statistics \cite{dumitriu2013functional,johnson2014cycles,ganguly2020random}. There is also a direct connection between random $d$-regular digraphs and random bipartite $d$-regular graphs. In particular, the eigenvalue fluctuation results of random bipartite regular graphs shown in \cite{dumitriu2020global} can be translated to  singular value fluctuations of random regular digraphs. However, it is a  very challenging problem to study the eigenvalue fluctuations of random regular digraphs.

\subsection*{Acknowledgements}
G.L. and Y.Z. acknowledge support from NSF  DMS-1928930 during their participation in the program ``Universality and Integrability in Random Matrix
Theory and Interacting Particle Systems'' hosted by the Mathematical Sciences Research Institute
in Berkeley, California during the Fall semester of 2021. 
G.L. is supported by the SNSF Ambizione grant S-71114-05-01.
Y.Z. is partially supported by  NSF-Simons Research Collaborations on the Mathematical and Scientific Foundations of Deep Learning. 
S.C. is supported by ENS-PSL.

\section{Main Results}\label{sec:results}

\subsection{Sums of random permutations}
For any $n\in\N$, let $d=d(n) \in\N$ and  $P^{(1)}, \dotsc, P^{(d)}$ be a collection of i.i.d.random uniform $n\times n$ permutation matrices. 
We consider the random matrix  
$$A_n : = P^{(1)} + \dotsb + P^{(d)}.$$ In the sequel, we will often drop the $n$ subscript and simply write $d, A$, etc. We define the (rescaled) characteristic polynomial of $A$,
\begin{equation} \label{charpoly}
\widehat\chi_n(z) := \frac{1}{\sqrt{d}}\det\left(I_n - z \frac{A}{\sqrt{d}}\right) \qquad ( z\in\C).
\end{equation}
Note that $A$ has a trivial eigenvalue $d$ associated with the vector $\mathbf{1}^n$, so that with our convention, $\widehat\chi_n(1/\sqrt d) =0$ almost surely. If fact, we will see that for any $n\in\N$,
\begin{equation}\label{E:chihat}
\E[ \widehat\chi_n(z) ] =  z-1/\sqrt{d} \qquad (z\in\C).
\end{equation}
The main results of this paper identify the weak $n\to \infty$ limit of the sequence $(\widehat{\chi}_n)$ in the unit disk either when the degree $d(n)=d$ is fixed, or when $d(n)\to\infty$ slowly with $n$. Up to a rescaling, the limit is the same as the function defined in \cite{coste2021sparse}; but we adopt different conventions allowing for a unified treatment of the cases $d=O(1)$ and $d\to \infty$.

We denote  $D_r = \{z\in\C :|z|<r\}$ the disk of radius $r>0$ in the complex plane.  {Let $H(D_1)$ be the set of analytic functions on $D_1$, equipped with the topology of uniform convergence on compact sets on $D_1$.}

\begin{definition}\label{def:1}
Fix $d\in\N$ and let
$\{\Lambda_\ell \}_ {\ell \in \mathbb{N}}$ be a family of independent random variables, with 
$$\Lambda_\ell \sim \mathrm{Poisson}\left(\frac{d^\ell}{\ell}\right) .$$
We also denote $\overline{\Lambda_\ell} = \Lambda_\ell-\E[ \Lambda_\ell] $ their centered versions. For any integer $d$, we define the following formal series:
\begin{equation} \label{Xd}
  \begin{aligned}
  & Y_d(z) := \sum_{k\in\N} \frac{z^k}{k d^{k/2}}  \sum_{\ell | k} \ell \overline{\Lambda_\ell} , \qquad
  & X_d(z) := \sum_{k\in\N} \frac{z^k}{d^{k/2}}  \overline{\Lambda_k} . 
  \end{aligned}
  \end{equation}
\end{definition}

Proposition \ref{prop:PAF} implies that these functions are well-defined and analytic functions in the unit disk $D_1$. Let us now state our main result for $d$ fixed. 

\begin{theorem}\label{thm:mc}
Consider a fixed integer $d$ and let $d(n)=d$ for every $n$; let $Y_d$ be as in \eqref{Xd}. Then,  
\[
\widehat\chi_n(z) \xrightarrow[n \to \infty]{\mathrm{law}} (z-1/\sqrt{d})   \frac{e^{-Y_d(z)} }{\E [e^{-Y_d(z)}]} 
\]
for the topology of uniform convergence on compact sets of $D_1$. 
\end{theorem}

In contrast, if the degree $d$ diverges, then we recover a Gaussian analytic function in the limit, as for dense matrices non-Hermitian (real-valued) Wigner matrices \cite{bordenave2020convergence} --- note that this function is real-analytic, the coefficients being \emph{real} Gaussian variables, not complex ones. The proof of Theorem \ref{thm:mc} and the subsequent results is outlined in Section \ref{sec:outline}.

\begin{definition}\label{def:2}
  Let $\{N_\ell\}_{\ell \in \mathbb N}$ be i.i.d.~standard real Gaussian  random variables and  define the random real-analytic function 
  \begin{equation} \label{Xinfty}
  X_\infty(z) =  \sum_{k\in\N} \frac{N_k}{\sqrt{k}} z^k  \qquad (z\in  D_{1} ) . 
  \end{equation}
  \end{definition}

\begin{theorem} \label{thm:gmc}
Consider a sequence $d(n)\in\N$ such that $d(n) \to\infty$  and $d(n) = n^{o(1)}$ as $n\to\infty$. 
Let $X_\infty$ be as in \eqref{Xinfty}.
Then, it holds for the topology of locally uniform convergence on~$D_{1}$,
\[
\widehat\chi_n(z) \xrightarrow[n \to \infty]{\mathrm{law}}  z  \sqrt{1-z^2} e^{X_\infty(z)} . 
\]
\end{theorem}

The technical condition $d(n) = n^{o(1)}$ means that $\log d(n) =o(\log n)$ as $n\to\infty$.

\subsection{Log-correlated structure of the limiting random fields} \label{sec:lcf}
We gather here a few properties of the functions $X_d,Y_d$.
In particular, we show that the boundary-values (on the unit circle $\partial  D_{1} $) of the functions $Y_d$ and $X_\infty$ are log-correlated fields and discuss some expected consequences.  {The exact correlation of $Y_d$ has an explicit expression (see Proposition 8.5 in \cite{coste2021sparse}), valid for every $d>0$}. The following proposition quantifies the difference between  $X_d$ and $Y_d$ when $d>1$. 

\begin{prop} \label{prop:PAF}
$ Y_d$ and $X_d$ are random (centered) real-analytic functions on $ D_1$. Moreover, 
 \begin{itemize}
 \item  For $z\in  D_{1}$ and $d\geq 1$, we have 
 \begin{equation} \label{expmd}
\E \left[ e^{-Y_d(z)} \right] = \bigg(  \prod_{\ell\in\N}  \mathrm{f}_\ell(z/\sqrt{d})^{\frac{d^\ell}{\ell}} \bigg)^{-1},
\end{equation}
 where $\mathrm{f}_\ell(z): = (1-z^\ell)e^{z^\ell}$. The infinite product converges uniformly on compact sets of $ D_{1}$. 
\item For $d\ge 2$, $Y_d = X_d + \Upsilon_d$ where 
$\E \left[ \|\Upsilon_d\|_{L^\infty( D_{1})}^{2} \right] \le C/d^{1/4}$ for a numerical constant $C$. Moreover, $X_d$ has covariance kernel 
\[ \E [ X_d(z) X_d(w)] 
 = \log(1-zw)^{-1}
\]
 and $Y_d$ has covariance kernel
 \[
\E [ Y_d(z) Y_d(w)] 
 = \log(1-zw)^{-1} +O_d(1) , \qquad z,w\in  D_{1},  
 \]
where the error term is a bi-analytic function which converges to $0$ uniformly in $\overline{ D_{1} \times   D_{1} } $ as $d\to\infty$.
 \end{itemize}
\end{prop}

 {Remarkably, the correlation kernel for $X_d$ is independent of $d\in\N$ and $z \in \partial D_{1}  \mapsto X_d(z)$ is a non-Gaussian log-correlated field, and if $d\ge 2$, so is the field $z \in \partial D_{1}  \mapsto Y_d(z)$.} 

One can check that for $d\in \mathbb{R}_+\cup\{\infty\}$, the random series $Y_d$ also converges in the Sobolev space of (Schwartz) distributions $H^{-\epsilon}(\partial D_1)$ of any $\epsilon>0$. 
In particular, for $d\in\N$, $\big(Y_d(z) : z \in \partial  D_{1}  \big)$ defines a non-Gaussian log-correlated field.
From this perspective, it is natural to ask for the asymptotics of the modulus of the limits in Theorems \ref{thm:mc}-\ref{thm:gmc}. Since $|\mathrm{e}^{z}| = \mathrm{e}^{\Re z}$, it is thus natural to study, for instance\footnote{$\big(Y_d(z) : z \in D_{1}  \big)$  is the harmonic extension of the log-correlated field inside the disk and this is a natural way to regularize this random generalized function. An alternative approach consists in truncating the Fourier series \eqref{Xd}, considering the asymptotics of the maximum of $\Re\big(\sum_{k\le N} \frac{z^k}{k d^{k/2}}  \sum_{\ell | k} \ell \overline{\Lambda_\ell}  \big)$ as $N\to\infty$. 
Theorem~\ref{thm:mc} shows that yet another regularization is given by the log characteristic polynomial $\log|\widehat\chi_n(z) |$ (after the appropriate centering).}, the behaviour of $\max_{|z|=r} \Re Y_d(z)$ as $r\to 1$.

If $d=1$, the leading order of the maximums of $Y_d$ and $\log|\widehat\chi_n|$ have been studied in  \cite{cook2020maximum}; specifically, it was shown that 
$\max_{z\in\partial D_1}\log|\widehat\chi_n| \sim x_0 \log n $ as $n\to\infty$ where $x_0\simeq 0.652$. 
The significance of this result is that the usual prediction for the value of $x_0$ coming from the theory of log-correlated fields \cite{FG15} does not apply to this problem and this is due to the tails of the field $\log|\widehat\chi_n|$ which are not Gaussian. 
A similar and more accessible question is the maximum of the characteristic polynomial of the CUE (circular unitary ensemble or Haar-distributed random matrices over the group U$(n)$). 
This problem has been thoroughly considered and  precise results about the asymptotics of the maximum are available \cite{FHK12,ABB17,PZ17,CMN18,L21}. 
We intend to consider the asymptotics for maximum of $Y_d$ and the leading order of that of $\log|\widehat\chi_n|$ for general $d$ in subsequent work. 
Another perspective is that $|\widehat\chi_n(z)|^\gamma \mathrm{d} z$ appropriately renormalized should converge to a multiplicative chaos\footnote{According to Theorem~\ref{thm:gmc}, the limiting random measures should be Gaussian multiplicative chaos (GMC) only in the regime as $d(n)\to\infty$. For fixed $d$, we still expect that the limiting random measures have similar multi-fractal properties \cite{DRSV17}.} measure as $n\to\infty$ for $\gamma>0$ sufficiently small (in the subcritical phase). In contrast to the CUE characteristic polynomial \cite{NSW20}, it is not clear whether the critical value is the standard $\gamma =\sqrt 2$ for this model or if it depends on the degree $d$. 

One can also investigate the regularity of the random (Schwartz) distributions $\big( e^{\sqrt \theta Y_d(z)} : z \in \partial  D_{1}  \big)$ depending on $\theta>0$.
This question has just been considered in the CUE case\footnote{The results of \cite{CMN18,L21,najnudel2020secular} are valid for general circular $\beta$-ensembles.} in \cite{najnudel2020secular} (for the so-called \emph{holomorphic multiplicative chaos}) which exhibits a phase transition related to the asymptotics for the maximum of the CUE field. 
In summary, these are fascinating research questions which motivates the study of characteristic polynomials of permutation matrices and the corresponding limiting random fields.

\bigskip

Finally, we relate the Gaussian analytic function $X_\infty$ in \eqref{Xinfty} with the $Y_d$ and $X_d$ defined before. It is well known (see \cite[Lemma 2.2.3]{hough2009zeros} for example) that $X_\infty$ is an analytic function over $D_1$, and a direct calculation shows that the  covariance kernel is given by
 \[
\E [X_\infty(z) X_\infty(w)] 
 = \log(1-zw)^{-1} ,\qquad z,w\in  D_{1} .
 \]

Thanks to the second point in Proposition \ref{prop:PAF}, we have the following convergence result. 

\begin{prop} \label{prop:GAF}
With the notation of Proposition~\ref{prop:PAF},
it holds for the topology of locally uniform convergence on $ D_{1}$, 
\[
( X_d ,   \Upsilon_d ) \xrightarrow[d \to \infty]{\mathrm{law}} (X_\infty,0) .
\]
In particular, 
\[
Y_d \xrightarrow[d \to \infty]{\mathrm{law}} X_\infty .
\]
\end{prop}

The proofs of Propositions~\ref{prop:PAF} and~\ref{prop:GAF} will be given in Section~\ref{sec:PAF}.
 We also refer to \cite[Section 2.3]{coste2021sparse} for the computation of the generating function of the exponential moments of $Y_d$.

\subsection{On fluctuations outside the support of the oriented Kesten-McKay density}The empirical spectral density of sums of $d$ random permutation matrices or random $d$-regular digraphs is conjectured to converge towards the oriented Kesten-McKay law, whose density with respect to the Lebesgue measure on $\mathbb{C}$ is 
$$\varrho_d(z) = \frac{d^2(d-1)}{\pi}\frac{\mathbf{1}_{|z|<\sqrt{d}}}{(d^2 - |z|^2)^2}.$$
The reader will find many insights regarding this problem in \cite{cook2019circular,metz2019spectral}.
The measure $\varrho_d$ is the Brown measure for the free sum of $d$ Haar unitary
operators, and the oriented Kesten-McKay law was established in \cite{basak2013limiting} for sums of $d$ independent Haar unitary or orthogonal matrices. 
 Our main result gives some information on the fluctuations around this law, in the spirit of \cite{rider2007noise,lambert2020maximum}.
The log-potential fluctuations around the oriented Kesten-McKay distribution are defined as the function
$$ \Psi_n(z) = \log|\det(zI_n - A_n)| - n U_d(z),$$
where $U_d$ is the Log-potential of $\varrho_d$, that is $U_d(z) = \int \log |z - x| \varrho_d(x)\mathrm{d}x$. Indeed, it can be checked by a direct computation that 
\begin{equation}\label{eq:UOKMC} U_d(z) = \begin{cases}
  \log |z|~~~~ \text{ for } |z|>\sqrt{d}, \\ -(d-1)\log \sqrt{d^2 - |z|^2} + \alpha_d \text{ for }|z|\leqslant \sqrt{d}
\end{cases}\end{equation}
where $\alpha_d=(d-1)\log \sqrt{1 - d^{-1}} + (d-1/2)\log(d)$.

 As a consequence, for every $|z|>\sqrt{d}$, the fluctuations are given by $\Psi_n(z) = \log |\det(I - z^{-1}A_n)|$ since the $n\log|z|$ terms cancel in $\Psi_n$. Our main result, Theorem \ref{thm:mc}, identifies the limit of the fluctuations $\Psi_n$ outside the disk $D_{\sqrt{d}}$.  
 In figure \ref{fig:3d} depicting $\Psi_n$, this corresponds to the smooth part of the picture. The rough part corresponds to the fluctuations inside the bulk of the oriented Kesten-McKay density; it is not clear what is going to be the correct definition of this generalized function.

\begin{figure}[H]
  \centering
  \begin{tabular}{cc}
  \includegraphics[width=0.6\textwidth]{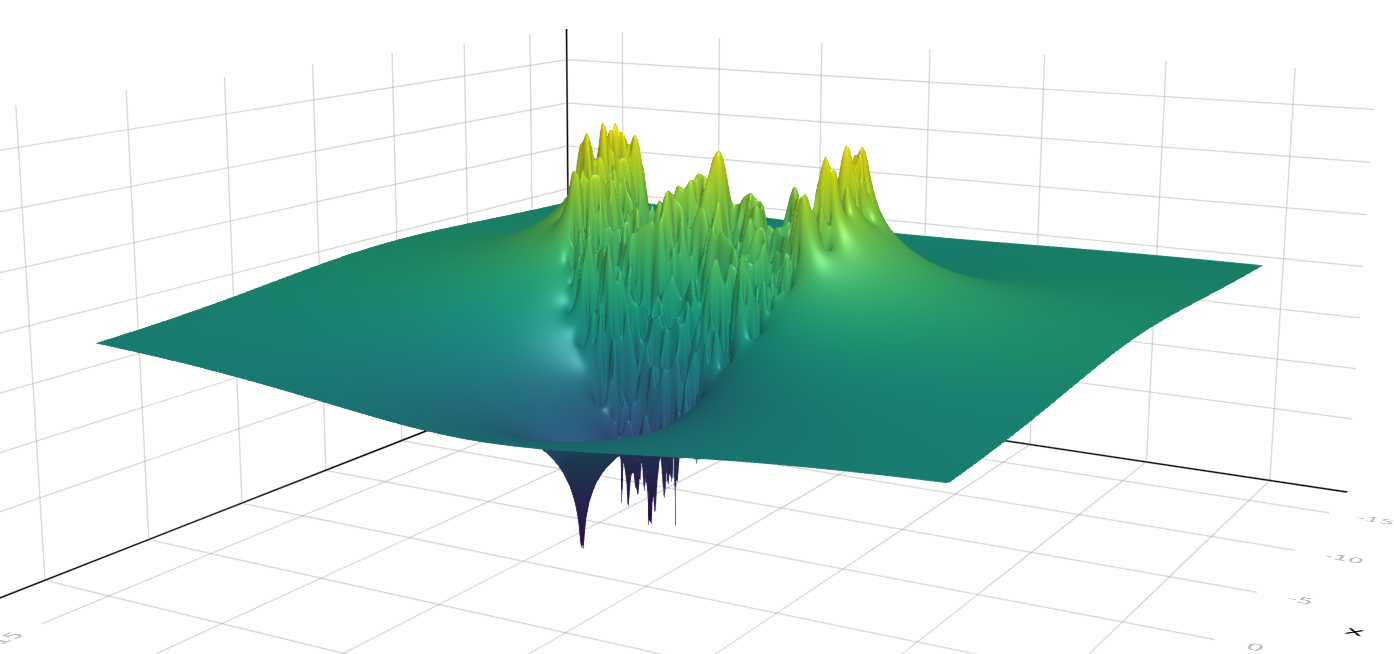}&
  \includegraphics[width=0.25\textwidth]{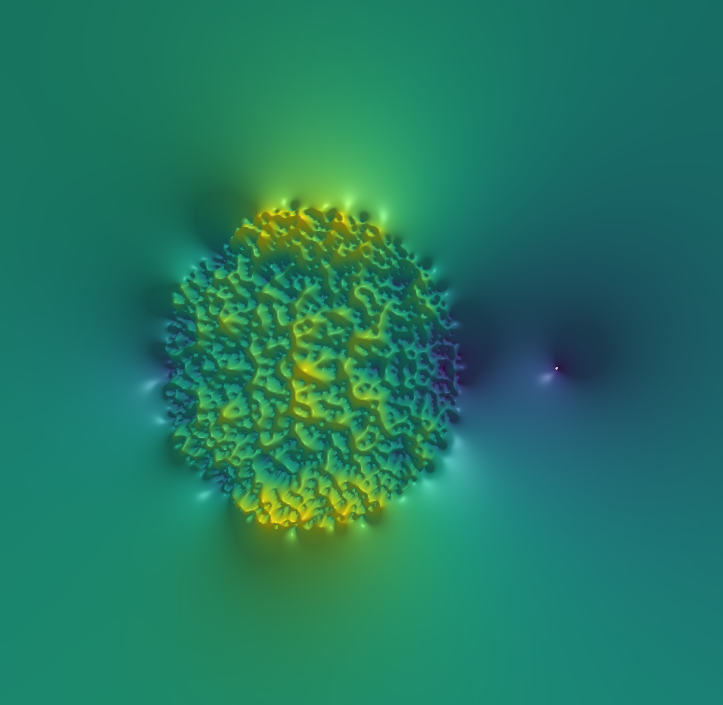}
  \end{tabular}
  \caption{Picture of the values of $\log |\det(z - A_n)| - n U_d(z)$ for $z$ in $[-5, 5]^2$, from two different angles; here, $A_n$ is a sum of $3$ uniform permutation matrices, $n=2000$ and $U_d$ is the Log-potential \eqref{eq:UOKMC}. The logarithmic singularity of $\Psi_n(z)$ at $z=d$ is visible in the smooth (harmonic) part of the picture. }\label{fig:3d}
\end{figure}

\section{Outline of the paper}\label{sec:outline}

\subsection{Proof strategy}

The proofs of Theorems~\ref{thm:mc} and~\ref{thm:gmc} are given in Section~\ref{sec:cvg}. 
They are based on the following three steps strategy;
\begin{itemize}
\item[(i)] If $(d(n))_{n\in\N}$ is a sequence of degree such that $d(n)= o(\sqrt n)$, then the family of random analytic function  $\{\widehat\chi_n(z) : z\in D_1 \}_{n\in\N}$  is tight. 
\item[(ii)] By expanding \eqref{charpoly}, for the principal branch of $\log$, it holds for $k\in \N$ that
\begin{equation} \label{logcharpoly}
[z^k]\log\bigg(\frac{\widehat\chi_n(z)}{z-1/\sqrt d} \bigg) = \frac{(-1)^{k+1}}{k} \frac{\tr(A^k)-d^k}{d^{k/2}},
\end{equation}
where $[z^k]f(z)$ means the $k$-th coefficient of the analytic function $f(z)$. 
Note that it is natural to consider $\log$ of $z\mapsto \frac{\widehat\chi_n(z)}{z-1/\sqrt d} $ to account for the trivial root of $\widehat\chi_n$ at $1/\sqrt d \in D_1$ if $d\ge 2$.
We establish that
\[
\frac{\tr(A^k)-d^k}{d^{k/2}} \xrightarrow[n \to \infty]{\mathrm{law}} L_k
\]
in the sense of finite dimensional distributions for a sequence of independent random variables $(L_k)_{k\in\N}$. This is the content of Theorem \ref{thm:trace} thereafter, which is proved in Section \ref{sec:trace} at page \pageref{sec:trace}.
\item[(iii)] 
Suppose that $\sum_{k\in\N} |L_k| r^k $ converges for any $r<1$ almost surely. 
From (i) tightness and (ii) we conclude that the random analytic function 
\[
\bigg(\frac{\widehat\chi_n(z)}{z-1/\sqrt d} \bigg) \xrightarrow[n \to \infty]{\mathrm{law}} \exp\bigg( \sum_{k\in \N} \frac{(-1)^{k+1}}{k} L_k z^k \bigg)
\]
in  the topology of locally uniform convergence on~$D_{1}$; cf.~Section~\ref{sec:cvg}. 
\end{itemize} 

Step (i) is somehow technical and we state it as a result. The outline of the proof and the key elements are in Section~\ref{sec:tight} at page \pageref{sec:tight}.  {We expect Proposition \ref{prop:tightness} to hold for larger $d(n)$ but this would require a different approach.}

\begin{prop}\label{prop:tightness}
  For any sequence $d(n) \in \N$ with $d = o(\sqrt{n})$, the sequence of random analytic function $(\widehat\chi_n)_{n\in\N}$  is a tight sequence for the compact-open topology. 
  \end{prop}

Step (ii) consists in a standard application of the trace method, without resorting to high traces. The main result is the following one.  {It is the same result as \cite{coste2021sparse}  (for different models) and it is reminiscent of the result for non-backtracking matrices of random regular graphs in \cite{dumitriu2013functional,johnson2015exchangeable}. The $d=1$ case is well known \cite{ben2015fluctuations,arratia1992cycle}.}

\begin{theorem}\label{thm:trace}
\begin{enumerate}[(1)]
\item Suppose that $d\in\N$ is fixed and $\{\Lambda_\ell \}_ {\ell \in \mathbb{N}}$ be as in Definition~\ref{def:1}.
Then for every $k\in\N$, 
\begin{equation}\label{eq:fix_convergence}
\big(\tr(A), \dotsc, \tr(A^k)\big) \xrightarrow[n \to \infty]{\mathrm{law}}\bigg(\Lambda_1, \dotsc, \sum_{\ell | k} \ell \Lambda_\ell \bigg).
\end{equation}
\item Consider a sequence $d(n)\in\N$ such that $d(n) \to\infty$  and 
$d(n) = n^{o(1)}$ as $n\to\infty$. Then for every $k\in\N$, 
\begin{align}\label{eq:grow_convergence}
\left(\frac{\tr(A)-d}{\sqrt d}, \dotsc, \frac{\tr(A^k)- d^k}{\sqrt{d^k}}\right) \xrightarrow[n \to \infty]{\mathrm{law}}\left(N_1, \sqrt{2} N_2+1,\dotsc,\sqrt{k} N_k +\mathbf{1}_{\{k \text{ is even}\}}\right) .
\end{align}
\end{enumerate}
\end{theorem}

The proof of Theorem~\ref{thm:trace} relies on the relationship between the random variables $\tr(A^k)$ and $k$-cycles on the (adjacency) digraph associated to the matrix $A$ -- note that $A$ is generically non-Hermitan, and $d$ is the degree of the regular digraph. 
The argument proceeds by the moment method, estimating the probabilities to observe some given collection of cycles for large $n$.  {Our proof remains valid for $k$ growing with $n$ but very slowly ($k = O(\log \log n))$; the approximation is expected to be true for faster rates for $k$, but would require more precise combinatorial arguments.}

\begin{remark}
Theorem \ref{thm:trace} yields the fluctuations of linear statistics  $\sum_{i=1}^n [f(\lambda_i)-\mathbb E f(\lambda_i)]$ for any (analytic) polynomial $f$, where $\lambda_1,\dots, \lambda_n$ are eigenvalues of $A/\sqrt{d}$. To study the fluctuations of linear statistics for general smooth test function with the moment method, one needs to extend the results to polynomials in $z$ and $\overline z$. This approach was developed in \cite{rider2007noise} for the Ginibre ensemble  {based on the determinantal structure}. 
\end{remark}

\subsection{Notational preliminaries}
\label{nota}

We never indicate that $d$ and the random variables defined in terms of $A$
depend on $n$=size($A$). 
We always consider limits as $n \to \infty$. 

\emph{Sets.} 
 For any $k\in\N$, denote $[k] =\{1, \dotsc, k\}$. 
If $I \subset \N$ is a finite set, we denote by $S_I$, the group of permutations of elements in $I$.
For $\sigma \in S_I$, we let $\epsilon(\sigma)$ be the sign of this permutation.
 If $I$ is a finite set, we denote by $|I|$  its cardinal (number of elements). Tuples are denoted by {\bf boldface letters}.
If $I \subseteq \N$ is a set, then \[I^k = \{ \mathbf{i} = (i_1,\dots,i_k) : i_1,\dots, i_k \in I\}\] for $k\in\N$. 
 For $k\in\N$, we denote
 \[\mathcal{E}_k = \{\mathbf{i} \in [n]^k :  i_1,\dots, i_k \text{ are distinct} \}.\]

\emph{Matrices.} In the sequel $A =(A_{i,j})_{i,j\in[n]}$ is a random matrix and we denote for any multi-index $\mathbf i \in [n]^k$,
\[
A_\mathbf{i} := A_{i_1, i_2} \dotsb A_{i_{k-1}, i_k}A_{i_k, i_1}.
\]
With this notation, 
\[
\mathrm{tr}(A^k) = \sum_{\mathbf{i} \in  [n]^k}A_{\mathbf{i}}.
\]
For any subset $I \subset [n]$, we denote
$A(I) = (A_{i,j})_{i,j\in I}$ the corresponding submatrix.

\emph{Directed graphs.} The matrix $A$ is an $n \times n$ matrix with integer entries. It represents the adjacency matrix of a { (weighted) digraph}, on the vertex set $V=[n]$. The edge set is { $E = \{(i,j) \in [n] \times [n] : A_{i,j} \ge 1 \}$ and the weights correspond to the entries of $A$}. We insist on the fact that $G_A=(V,E)$ is directed and possibly has simple loops.
We can interpret  $\mathbf i \in [n]^k$ as  \emph{cyclic paths} on the graph of $A$ and $A_\mathbf{i}$ as the \emph{weight} of this path.  
In addition, the paths $\mathbf i \in \mathcal{E}_k$ correspond to simple loops.

\emph{Random permutation.}
We define a random permutation matrix $ P =  (\1_{\pi(i)=j})_{i,j\in[n]}$, where $\pi$ is a uniform random element of $S_{[n]}$. 
We will use that under this law, for any $k\in[n]$, any $\mathbf i  , \mathbf j \in \mathcal E_k$,
\[
\P[ \pi(i_1)=j_1, \dotsc, \pi(i_k)=j_k ] = \P[ P_{i_1,j_1} = \cdots = P_{i_k,j_k} =1] =    \frac{(n-k)!}{n!} .
\]
We will also make several use of the following bounds; for any $d\in\N$, if $ \mathbf k \in \N^d$ with $k_1+ \cdots + k_d = k \le n$, then
\begin{equation} \label{factorbd1}
\prod_{\delta=1}^d \frac{(n-k_{\delta})!}{n!}\leq \frac{(n-k)!}{n!} .
\end{equation}

\section{Tools for proving the tightness of $\{\chi_n\}$}

\subsection{A simple tightness criterion and proof for Proposition \ref{prop:tightness}} \label{sec:tight}

 {By Montel's theorem}, a family of analytic functions $\{f_n\}$ on a domain $D$ is tight (for the compact-open topology) if for any compact $K\subset D$, 
\[
\big\{ \|f_n\|_{L^\infty(K)} \big\} \text{ is tight.}
\] 
 {See, for example, \cite[Proposition 2.5]{shirai2012limit}.
 In particular, by Markov's inequality}, it suffices to check that for  {some} $\alpha>0$
 \begin{align}\label{eq:unif_moment_bound}
 \sup_{n} \E  \|f_n\|_{L^\infty(K)}^{\alpha} <\infty .
 \end{align}
 
We will use this criterion to show that the characteristic polynomial of a random $n\times n$ matrix $A$,
\[\big\{ \chi_n = \det(I+ \cdot A) \big\} \]  forms a tight sequence  in a disk  $D_r =\{|z| < r\}$.
Note that the random matrices $A$ need not be defined on the same probability spaces for different $n\in\N$.  Recall that for $z\in\C$,
\[
\chi_n(z) =\sum_{k\le n} z^k  \Delta_k^{(n)} ,
\]
where $\Delta_k$ are called \textit{secular coefficients} and are given by $\Delta_0 =1$, 
\begin{align} \label{seccoeff}
    \Delta_k^{(n)}=\sum_{I\subseteq [n] , |I|=k} \det (A(I)) , \qquad k \in [n].  
\end{align}
 {The key technical element for the proof of Proposition \ref{prop:tightness} is the following estimate on the secular coefficients. 
\begin{lemma}\label{lemma:technical}Let $A=\sum_{q=1}^d P^{(q)}$ where $P^{(q)}$ are $d\in\N$ independent random permutation matrices. 
  For any $n\geq 2$, it holds for all $0<r<1$ and $1\leq d<\sqrt{\frac{n(1-r)}{r}}$, 
  \begin{align}\label{eq:moment_secular}
  \sum_{k\le n} \frac{r^k}{d^{k+1}} \E  |\Delta_k^{(n)}|^2
   \le  \frac{\frac{2}{d}+r }{(1-r - \frac{rd^2}{n})^2}. 	   
  \end{align}
\end{lemma}
The proof of this lemma is postponed to a subsequent section.}

\begin{proof}[Proof of Proposition \ref{prop:tightness} given \eqref{eq:moment_secular}]
For any $\theta<1$, $r>0$ and $\epsilon>0$,  {Jensen's inequality yields
\[
  \bigg(\sum_{k = 0}^n r^k \theta^k  |\Delta_k^{(n)}| \bigg)^{1+\epsilon} \le \frac{1}{(1-\theta)^\epsilon} \sum_{k = 0}^n r^{k(1+\epsilon)}   |\Delta_k^{(n)}|^{1+\epsilon}
\]
and consequently, 
\[
 \E  \|\chi_n\|_{L^\infty(D_{r\theta})}^{1+\epsilon} 
 \le \E  \bigg(\sum_{k = 0}^n r^k \theta^k  |\Delta_k^{(n)}| \bigg)^{1+\epsilon} \le \frac{1}{(1-\theta)^\epsilon} \sum_{k = 0}^n r^{k(1+\epsilon)} \E  |\Delta_k^{(n)}|^{1+\epsilon}.
\]
}

With the normalization \eqref{charpoly}, by scaling and choosing $\epsilon=1$, 
it holds for $r>0$ and $\theta<1$, 
\begin{align}\label{eq:1+eps_moment}
 \E  \|\widehat\chi_n\|_{L^\infty(D_{r\theta})}^{2} 
  \le \frac{1}{1-\theta} \sum_{k=0}^n \frac{r^{2k}}{d^{k+1}} \E  |\Delta_k^{(n)}|^{2}.
\end{align}

Now, \eqref{eq:unif_moment_bound}, \eqref{eq:1+eps_moment}, and \eqref{eq:moment_secular} immediately imply Proposition \ref{prop:tightness}.

\end{proof}

\subsection{Exchangeability}

By \eqref{seccoeff}, it holds for any $k\in[n]$
\[
    \mathbb E  \Delta_k^2 =\sum_{|I|, |J|=k} \mathbb E\det (A(I)) \det(A(J)) .
\]
where the sum is over all subsets $I, J\subseteq [n] $. 
Lemma~\ref{lem:reduction} below allows to reduce this sum to subsets $(I,J)$ which share at least $k-1$ elements;
\begin{equation} \label{reduction}
  \mathbb \E \Delta_k^2 =   \sum_{|I|=k} \mathbb E\det (A(I))^2 +
 \sum_{|I|=|J|=k, |I\cap J|=k-1}  \mathbb E\det (A(I)) \det(A(J)) .
\end{equation}

The Lemmas of this section apply to any  exchangeable random matrix model, that is, if the following invariance property holds; for any given permutation $\pi \in S_{[n]}$,
\begin{equation} \label{exchange}
  \big(A_{i,\pi(j)} \big)_{i,j\in[n]} \overset{\rm law}{=} \big(A_{\pi(i),j} \big)_{i,j\in[n]}   \overset{\rm law}{=}     \big(A_{i,j} \big)_{i,j\in[n]} .
\end{equation}

\begin{lemma}  \label{lem:expDelta}
For any $k\in [n]$,
\[
\E \Delta_k^{(n)} =  \1_{k=1} \E \tr A .
\]

\end{lemma}

In particular, for the random matrix $A =\sum_{q=1}^d P^{(q)}$, 
we have $\E \Delta_1^{(n)}= d$  {and $\E \Delta_k^{(n)}=0$ for $k>1$, hence \eqref{E:chihat} holds.}

\begin{proof}
By \eqref{seccoeff}, it holds for any $k\in[n]$
\[
\E\Delta_k =\sum_{|I|=k} \E\det (A(I)) .
\]
If the random matrix $A$ satisfies \eqref{exchange}, observe that  by permuting two columns of $A$, $\E\det (A(I)) = - \E \det (A(I))$, hence $\E \det(A(I))=0$ for any $I \subset [n]$ provided that $|I|\ge 2$. 
This shows that $\E \Delta_k =0$ for all $k\ge 2$ and 
\[
\E \Delta_1 = \sum_{i\in [n]} \E A_{i,i} = \E \tr A . \qedhere
\]
\end{proof}

\begin{lemma} \label{lem:reduction}
Assume $|I|=|J|=k$. If $|I\cap J|\leq k-2$, then  $ \mathbb E\det (A(I)) \det(A(J))=0$.
\begin{proof} For ease of notation, we consider the case when $|I\cap J|=k-2$. Other cases follow in the same way. 
Let us denote $\{x,y\} = I\setminus J$. 
Given $\sigma \in S_I , \tau \in S_J$, using the invariance in law \eqref{exchange} applied to $\pi = (xy)$, we have 
\begin{align*}
\mathbb E\left[ \prod_{i\in I,j\in J} A_{i,\sigma(i)} A_{j,\tau(j)}\right]
= \mathbb E\left[ \prod_{i\in I,j\in J} A_{i,\sigma'(i)} A_{j,\tau(j)}\right]
\end{align*}
where $\sigma' = \sigma \circ \pi$. Note that we used that $\tau  \circ \pi = \tau$ since $x, y \notin J$. Since the map $\sigma\to \sigma' =  \sigma \circ \pi$ is a bijection on $S_I$ with the property that $\epsilon(\sigma')=-\epsilon(\sigma)$, by summing over all $\sigma\in S_I$, this  implies that for any fixed $\tau\in S_J$,
\begin{align*}
    \sum_{\sigma \in S_I}\epsilon(\sigma)\epsilon(\tau) \mathbb E\left[ \prod_{i\in I,j\in J} A_{i,\sigma(i)} A_{j,\tau(j)}\right] =0.
\end{align*}
Summing over $\tau \in S_J$ gives the desired result. 
\end{proof}
\end{lemma}

From Lemma \ref{lem:reduction},  identity \eqref{reduction} holds.

\section{The general case: tightness and trace asymptotics}
\label{sec:trace}

\subsection{Proving the tightness of the determinants in the $d=1$ case}

When $d=1$, the matrix $A$ is indeed the permutation matrix of a uniform permutation of $[n]$ and the analysis is easier to perform. We include this special case because it gives all the ideas for the general proof.

When $d=1$, the whole problem reduces to the study of random permutation matrix. In this case we can calculate the second moment of secular coefficients exactly,  {which gives a proof of Lemma \ref{lemma:technical} for the case $d=1$.} 
 
\begin{lemma}For all $1\leq k\leq n-1$,
\begin{align} \label{eq:deltak}
  \mathbb E |\Delta_k|^2=2,
\end{align}
and 
\begin{align}
    \mathbb E |\Delta_n|^2=1. \notag 
\end{align}
\end{lemma}
\begin{proof}Let $\pi \in S_n $ be a uniform random permutation and $A$ be  the permutation matrix of $\pi$.
By Lemma \ref{lem:reduction}, we only need to evaluate
\begin{align*}
    \sum_{|I|=k}\mathbb E\det(A(I)^2)+\sum_{|I|=|J|=k, |I\cap J|=k-1} \mathbb E\det A(I)A(J).
\end{align*}
The summation in the second term is non-empty if and only if $k<n$. For the first term,
\begin{align*}
 \mathbb E\det(A(I)^2) &=\sum_{\sigma,\tau\in S(I)}\epsilon(\sigma)\epsilon(\tau)  \mathbb E\prod_{i\in I}A_{i,\sigma(i)}A_{i,\tau(i)}\\
  &=\sum_{\sigma,\tau} \epsilon(\sigma)\epsilon(\tau) \mathbb P\left( \pi(i)=\sigma(i)=\tau(i),\forall i\in I \right)\\
  &=\sum_{\sigma}  \mathbb P\left( \pi(i)=\sigma(i), \forall i\in I\right)=\mathbb P(\pi: I\to I)=\frac{k!(n-k)!}{n!}=\frac{1}{\binom{n}{k}},
\end{align*}
where $\pi: I\to I$ is the event that $\pi$ restricted on $I$ is a permutation on $I$.
Hence 
\begin{align}
     \sum_{|I|=k}\mathbb E\det(A(I)^2)=1. \notag 
\end{align}
By taking $k=n$, it gives $\mathbb E[|\Delta_n|^2]=1$.

For the second term, with the assumption $|I|=|J|=k, |I\cap J|=k-1$,
\begin{align}
    \mathbb E\det A(I)A(J)&=\sum_{\sigma\in S(I),\tau\in S(J)}\epsilon(\sigma)\epsilon(\tau)\mathbb E\prod_{i\in I} A_{i,\sigma(i)}\prod_{j\in J} A_{j,\tau(j)}\notag \\
    &=\sum_{\sigma\in S(I),\tau\in S(J)}\epsilon(\sigma)\epsilon(\tau)  \mathbb P(\pi(i)=\sigma(i), i\in I, \pi(j)=\tau(j), j\in J)\notag 
\end{align}
 Without loss of generality, we assume $I=\{1,\dots, k\}, J=\{2,\dots,k+1\}$. To have a non-zero probability, we must have $\sigma(i)= \tau(i) $ for all $i\in I\cap J$. This forces $\sigma,\tau$ to be two permutations on $I\cap J$, and $\sigma(1)=1,\tau(k+1)=k+1$. Let 
\[\Omega=\{(\sigma,\tau)\in S(I)\times S(J): \sigma(1)=1, \tau(k+1)=k+1,\sigma(k)=\tau(k), \forall k\in I\cap J \}. 
\]
For any $(\sigma,\tau) \in \Omega$, both permutations have the same cycle types, hence $\epsilon(\sigma)=\epsilon(\tau)$.
Therefore 
\begin{align*}
    \mathbb E\det A(I)A(J)
    &=\sum_{(\sigma,\tau)\in\Omega } \mathbb P(\pi(i)=\sigma(i), i\in I, \pi(j)=\tau(j), j\in J)\\
    &=\mathbb P(\pi(1)=1,\pi(k+1)=k+1, \pi: \{2,\dots, k\}\to \{2,\dots, k\})\\
    &= \frac{(k-1)!(n-k-1)!}{n!}.
\end{align*}
Since $\#\big\{ I,J \subset [n]  : |I|=|J|=k, |I\cap J|=k-1 \big\} = \binom{n}{k}k(n-k)$, we conclude that 
\[
    \sum_{|I|=|J|=k, |I\cap J|=k-1} \mathbb E\det A(I)A(J)=\binom{n}{k}\frac{k!(n-k)!}{n!}=1.
\]
\end{proof}

\subsection{Proving the tightness of the sequence $(\chi_n)$ with the second moment argument}

In the case of a sum of $d>1$ permutations, proving the tightness of $(\chi_n)$ using the criteria given in Subsection \ref{sec:tight} is more involved and necessitates a detailed technical analysis of the determinant expansion of $I - zA$. We recall that $A=\sum_{q=1}^d P^{(q)}$, where $P^{(q)}, q\in[d]$ are $n\times n$ independent random permutation matrices. 
The subsequent computations crucially rely on the exchangeability of $A$ as in \eqref{exchange}.

\begin{lemma} \label{lem:1}
For any $n\geq 2$, it holds for all $0<r<1$ and $1\leq d<\sqrt{\frac{n(1-r)}{r}}$,
\[
\sum_{k\geq 1} \frac{r^k}{d^k} \sum_{|I|=k}   \mathbb E \det (A(I)^2) 
 \le \frac{2}{1-r-\frac{rd^2}{n}}  . 
\]
\end{lemma}

\begin{proof}
First observe that 
\[
 \sum_{|I|=k}\mathbb E\det(A(I)^2)  
=  \sum_{|I|=k} \sum_{\sigma,\tau\in S(I)}\epsilon(\sigma)\epsilon(\tau) \mathbb E\prod_{i\in I}  A_{i,\sigma (i)}\prod_{j\in I} A_{j,\tau(j)} .
\]
Then, by symmetry of the permutation model, 
\begin{equation}\label{exp1}
\begin{aligned}
 \sum_{|I|=k}    \mathbb E \det (A(I)^2)    
 &=\binom{n}{k}\sum_{\sigma,\tau\in S_k} \epsilon(\sigma)\epsilon(\tau) \prod_{i\in[k]} \mathbb E A_{i,\sigma(i)}A_{i,\tau(i)}\\
    &=\binom{n}{k}\sum_{\sigma,\tau\in S_k}
\epsilon(\sigma^{-1}\tau) \E\prod_{i\in[k]} A_{i,i}A_{i,\sigma^{-1}\tau(i)}\\
    &=\binom{n}{k}k!\sum_{\tau\in S_k} \epsilon(\tau) \E \prod_{i\in[k]} A_{i,i} A_{i,\tau(i)}\\
    &=\frac{n!}{(n-k)!}\sum_{q_1,\dots,q_k\in [d]}\sum_{\ell_1,\dots,\ell_k\in [d]}\sum_{\tau\in S_k}\epsilon(\tau)\E \prod_{i\in[k]} P_{ii}^{(q_i)} P_{i\tau(i)}^{(\ell_i)} . 
 \end{aligned}
 \end{equation}
where we used \eqref{exchange} to obtain \eqref{exp1}.
Observe that for any fixed $q,\ell\in [d]^k$ and all  $i\in[k]$, if $q_i = \ell_i$ then
\[
P_{i,i}^{(q_i)} P_{i,\tau(i)}^{(\ell_i)}= 
P_{i,i}^{(q_i)} \1\{\tau(i)=i\} ,
\]
and using independence, 
\[ \begin{aligned}
\E \prod_{i\in[k]} P_{ii}^{(r)} P_{i\tau(i)}^{(r)}
&=\prod_{r\in[d]} \E \bigg[  { \prod_{i: q_i=r} P_{ii}^{(r)}  \prod_{i: \ell_i=r} P_{i\tau(i)}^{(r)}} \bigg] \\
&= \prod_{r\in[d]} \1\big\{\tau(i)= i ; \forall i \in \{q_i=\ell_i=r\}\big\} \E \bigg[  { \prod_{i: q_i=r} P_{ii}^{(r)}  \prod_{i\in K_r} P_{i\tau(i)}^{(r)} }\bigg]  
\end{aligned}\] 	
where $K_r = K_r(\ell,q) =\{i :   \ell_i=r , \ q_i \neq r \}$. 
Let $K=K(\ell,q) = \{i : q_i \neq \ell_i\}  =   \bigcup_{r=1}^dK_r$ (disjoint) and for $r\in[d]$, 
\[
 \Theta_r(\tau)= \E \bigg[  \prod_{i: q_i=r}  {P_{ii}^{(r)}  \prod_{i\in K_r} P^{(r)}_{i\tau(i)}} \bigg] \1\big\{\tau(i) \notin \{i: q_i=r\} ; \forall i\in K_r\big\} .\] 
This implies that 
\begin{equation} \label{Theta}
\sum_{\tau\in S_k}\epsilon(\tau)\E \prod_{i\in[k]} P_{ii}^{(q_i)} P_{i\tau(i)}^{(\ell_i)}
= \sum_{\tau\in S_K}\epsilon(\tau)  \prod_{r\in[d]} \Theta_r(\tau)
\end{equation}
 {where in the right hand side of \eqref{Theta}, a sum over an empty set is considered 0.} 
By \eqref{exchange}, observe that for any permutation $\sigma\in S_{K}$ which fixes the subsets $\{K_r\}_{r\in[d]}$, 
\[
\prod_{r\in[d]} \Theta_r(\sigma\tau)=   \prod_{r\in[d]} \Theta_r(\tau). 
\]
Let \[\delta_r=\#K_r =\# \{ i: \ell_i=r,q_i\not=\ell_i\}.\]
If $\delta_r\ge 2$ for some $r\in[d]$, by a transposition, one has
\[
\sum_{\tau\in S_K}\epsilon(\tau)  \prod_{r\in[d]} \Theta_r(\tau) =0. 
\]
In particular, this shows that for any fixed $q,\ell\in [d]^k$,
\[
\bigg| \sum_{\tau\in S_K}\epsilon(\tau)  \prod_{r\in[d]}   \Theta_r(\tau) \bigg| \le  \Delta!  \prod_{r\in[d]} \1\big\{\delta_r \in\{0,1\}\big\} \frac{(n-\#\{i: q_i=r\}-\delta_r)!}{n!} ,
\qquad
\Delta = \sum_{r\in[d]} \delta_r  = \# K. 
\]

Observe that for any integers $k_1\geq 1, k_2\geq 0$, $\delta_1,\delta_2\in \{0,1\}$,
\begin{align}\label{eq:deltainequality}
\frac{(n-k_1-\delta_1)!}{(n-k_1-k_2)!} \le 
\begin{cases}
\frac{n!}{(n-k_2-\delta_2)!}\frac{1}{n^{\delta_1+\delta_2}}, & \delta_1+\delta_2\leq 1,\\\frac{n!}{(n-k_2-\delta_2)!} \frac{2}{n^{\delta_1+\delta_2}}, & \delta_1=\delta_2=1.
\end{cases}
\end{align}
Let $k_r(q)=\#\{i:q_i=r\}$. We have $k_1+\cdots+k_r=k$. By symmetry, we can assume $k_1\geq 1$, and by induction, we obtain
\[\begin{aligned}
\prod_{r\in[d]}\frac{(n-k_r-\delta_r)!}{n!} 
&\le \frac{2}{n^{\delta_1+\delta_2}} \frac{(n-k_1-k_2)!}{n!}\prod_{r\ge3}\frac{(n-k_r-\delta_r)!}{n!}  \\
&\le \frac{(n-k)!}{n!} \frac{2}{n^\Delta}
\end{aligned}\]
where $\Delta = \sum_{r\in[d]} \delta_r$.
We conclude that 
\[
\bigg| \sum_{\tau\in S_K}\epsilon(\tau)  \prod_{r\in[d]}   \Theta_r(\tau) \bigg| \le  \frac{2\Delta!(n-k)!}{n!n^{\Delta}}
\prod_{r\in[d]} \1\big\{\delta_r \in\{0,1\}\big\} , 
\qquad
\Delta = \sum_{r\in[d]} \delta_r =\#K .
\]

Going back to \eqref{exp1}, using formula \eqref{Theta}, this argument shows that 
\[\begin{aligned}
 \sum_{|I|=k}   \mathbb E \det (A(I)^2) 
&\le 2 \sum_{q_1,\dots,q_k\in [d]}\sum_{\ell_1,\dots,\ell_k\in [d]} 
 \frac{\Delta!}{n^{\Delta}}  \prod_{r\in[d]} \1\big\{\delta_r \in\{0,1\}\big\} . 
\end{aligned}\]

For  $\ell\in [d]^k$, let us denote $k_r = k_r(\ell) = \#\{ i : \ell_i=r\}$ for $r\in [d]$. 
Observe that given $k_1,\dots,k_r$, there are $\binom{k}{k_1,\dots,k_d}$ many choices of $\ell\in [d]^k$, and for  $\delta \in \{0,1\}^d$, there are at most $d^\Delta \prod_{r=1}^d k_r^{\delta_r}$ configurations $q\in [d]^k$ which contributes to the previous sum.  
Hence, 
\begin{align}
 \sum_{|I|=k}   \mathbb E \det (A(I)^2) 
& \le 2  \sum_{k_1,\dots, k_d\ge 0}  \binom{k}{k_1,\cdots,k_d} \sum_{\delta_1,\dots,\delta_d\in \{0,1\}}  \prod_{r\in[d]} \bigg(\frac{k_r d^2}{n}\bigg)^{\delta_r} \notag\\
& = 2  \sum_{k_1,\dots, k_d\ge 0}  \binom{k}{k_1,\cdots,k_d}\prod_{r\in[d]} \bigg(1+ \frac{k_r d^2}{n}\bigg) \notag\\
&=2 d^k\big( 1+ \tfrac{d}{n} \partial_x \big)^d x^k\big|_{x=1} . \label{eq:derivative_sum}
 \end{align}
The last identity follows from the fact that the second display equals (up to a factor 2)
\[\begin{aligned}
\sum_{k_1,\dots, k_d\ge 0}  \binom{k}{k_1,\cdots,k_d}   \prod_{r\in [d]}  \big( 1+ \epsilon\partial_x \big)x^{k_r}  \big|_{x=1}
& =  \prod_{r\in [d]}  \big( 1+ \epsilon \partial_{x_r} \big) \bigg(
 \sum_{k_1,\dots, k_d\ge 0}  \binom{k}{k_1,\cdots,k_d}   \prod_{r\in [d]} x_r^{k_r}  \bigg) \bigg|_{ {x_1 = \dotsc = x_d=1}} \\
 & =  \prod_{r\in [d]}  \big( 1+ \epsilon \partial_{x_r} \big) \big(x_1+\cdots+ {x_d}\big)^k\big|_{ {x_1=\cdots=x_d=1}} \\
 & = \big( 1+ \epsilon \partial_x \big)^d x^k\big|_{x=d}  \\
 &=  d^k\big( 1+ \tfrac{\epsilon}{d} \partial_x \big)^d x^k\big|_{x=1}
\end{aligned}\]
by scaling, and taking $\epsilon=d^2/n$.
By summing over all $k\in\N_0$, we conclude that for any $0<r<1$, 
\[\begin{aligned}
\sum_{k\in\N} \frac{r^k}{d^k} \sum_{|I|=k}   \mathbb E \det (A(I)^2) 
 & \le 2\big( 1+ \tfrac{d}{n} \partial_x \big)^d  \frac{1}{1-xr} \bigg|_{x=1} \\
 & = 2 \sum_{\ell \ge 0} \binom{d}{\ell} \frac{d^\ell}{n^\ell} \partial_x^\ell \frac{1}{1-xr} \bigg|_{x=1} \\
&\le 2 \sum_{\ell=0}^d \frac{d!}{(d-\ell) !} \frac{(dr)^\ell}{n^\ell} \frac{1}{(1-r)^{\ell+1}}\\
& \le \frac{2}{1-r} \sum_{\ell=0}^{\infty} \left(\frac{d^2r}{n(1-r)}\right)^{\ell}\leq \frac{2}{1-r} \cdot\frac{1}{1-\frac{d^2r}{n(1-r)}} =\frac{2}{1-r(1+\frac{d^2}{n})}. 
 \end{aligned}\]
 The previous sum converges if $\frac{d^2 r}{n(1-r)}<1$.
This completes the proof. 
\end{proof}

\begin{lemma} \label{lem:2}
For any $n\geq 2,$  it holds for all $0<r<1$ and $1\leq d<\sqrt{\frac{n(1-r)}{r}}$,
\[
\sum_{k\geq 1} \frac{r^k}{d^{k+1}}
\bigg|  \sum_{|I|=|J|=k, |I\cap J|=k-1} \E\det A(I)A(J)\bigg| 
 \le \frac{r}{(1-r - \frac{rd^2}{n})^2}  . 
\]
\end{lemma}

\begin{proof}
Using the invariance \eqref{exchange} of the random matrix $A$, for $k\leq n-1$, we can rewrite
\begin{align}\label{eq:secondsecular}
\sum_{|I|=|J|=k, |I\cap J|=k-1} \E\det A(I)A(J)
& = \binom{n}{k} k (n-k)  \sum_{\substack{\tau\in S_k\otimes S_1 \\ \sigma\in S_1\otimes S_k}} \epsilon(\sigma)\epsilon(\tau) \E \prod_{j=1}^k \prod_{i=2}^{k+1} A_{i,\sigma(i)}A_{j,\tau(j)}
\end{align}
where $S_k\otimes S_1=\{ \tau\in S_{k+1}:\tau(k+1)=k+1\}.$ and  $S_1\otimes S_k=\{ \sigma\in S_{k+1}:\sigma(1)=1\}.$ Now \eqref{eq:secondsecular} can be written as
\begin{align*}
&=\binom{n}{k} k (n-k) \sum_{\substack{\tau\in S_k\otimes S_1 \\ \sigma\in S_1\otimes S_k}} \epsilon(\sigma)\epsilon(\tau) \sum_{\ell_1,\dots,\ell_k}\sum_{q_2,\dots,q_{k+1}}\E  \prod_{i=2}^{k+1}  P_{i,\sigma(i)}^{(q_i)}\prod_{j=1}^kP_{j,\tau(j)}^{(\ell_j)}\\
 &=\binom{n}{k} k (n-k)\sum_{\ell_1,\dots,\ell_k}\sum_{q_2,\dots,q_{k+1}} \sum_{\substack{\tau\in S_k\otimes S_1 \\ \sigma\in S_1\otimes S_k}} \epsilon(\sigma)\epsilon(\tau) \1_{\big\{\tau(j)=\sigma(j) , \forall j \in [k] \setminus K \big\}} \prod_{r\in[d]} \E \bigg[  \prod_{i \ge 2: q_i=r} P_{i,\sigma(i)}  \prod_{j\in K :  \ell_j=r  } P_{j\tau(j)} \bigg] , 
\end{align*}
where $K=K(\ell,q) = \{j \in [k] : q_j \neq \ell_j\}$.
By convention $1\in K$ and by symmetry we may assume that $l_1=1$ and multiply the sum with an extra factor $d$.
Moreover, by the argument following \eqref{Theta}, we must have for all $r\in\{2,\dots,d\}$, 
\[
\delta_r  =\#\{j\in K :  \ell_j=r \} \in \{0,1\}
\]
for otherwise the sum over all permutation $\tau\in S_k$ yields 0 by a transposition. Similarly, we must have
\[
\delta_1  =\#\{j\in K :  \ell_j=1 \} =1\]
since $1\in K$.  
Then, for any configuration $q,\ell\in [d]^k$ with $\ell_1=1$, following the  argument in \eqref{eq:deltainequality}, it holds for any $\sigma,\tau \in S_{k+1}$
\[ \begin{aligned}
\prod_{r\in[d]}  \E \bigg[  \prod_{i \ge 2: q_i=r} P_{i,\sigma(i)}  \prod_{j\in K :  \ell_j=r  } P_{j,\tau(j)} \bigg]     &  \le    \prod_{r\in[d]} \frac{(n-\#\{i \ge 2: q_i=r\}-\delta_r)!}{n!}  \\
&=  \frac{(n-k_1-1)!}{n!}
 \prod_{r\ge 2} \frac{(n-k_r-\delta_r)!}{n!}  \\
& \le \frac{(n-k-1)!}{n! \, n^{\Delta}}
\end{aligned}\]
where $\Delta =  \sum_{r=2}^{d} \delta_r$ and $k_r=k_r(q)= \#\{i \ge 2: q_i=r\}$ form a partition of $k$ for $r\in[d]$. 
Moreover, for any   $\Delta \le d-1$ and any given $\sigma\in S_1\otimes S_k$,
\[\#\big\{ \tau\in S_k : \tau(j)=\sigma(j) , \forall j \in [k] \setminus K  \big\} \le (\Delta+1)! \le  d^{\Delta},\]
so that this argument shows that 
\[\begin{aligned}
 &\binom{n}{k} k (n-k) \sum_{\ell_1,\dots,\ell_k}\sum_{q_2,\dots,q_{k+1}} \sum_{\substack{\tau\in S_k\otimes S_1 \\ \sigma\in S_1\otimes S_k}} \epsilon(\sigma)\epsilon(\tau) \1_{\big\{\tau(j)=\sigma(j) , \forall j \in [k] \setminus K \big\}} \prod_{r\in[d]} \E \bigg[  \prod_{i \ge 2: q_i=r} P_{i,\sigma(i)}  \prod_{j\in K :  \ell_j=r  } P_{j\tau(j)} \bigg]   
\\
&\le d k\sum_{q_2,\dots,q_{k+1}\in [d]}\sum_{l_2,\dots,l_{k}\in [d]}  \frac{d^{\Delta}}{n^\Delta}\1\big\{\delta_1 =1\big\} \prod_{r=2}^d \1\big\{\delta_r \in\{0,1\}\big\}. 
\end{aligned}\]

Let $k_r=\#\{ i\geq 2: \ell_i=r\}$ for $r\in [d]$.  There are $\binom{k-1}{k_1,\dots,k_d}$ many choices of $\ell\in [d]^k$ with $\ell_1=1$. 
As in the proof of Lemma~\ref{lem:1}, for a given $l\in [d]^k$ with $\ell_1=1$, there are at most 
\[ d^{\Delta+1} \prod_{r=2}^d k_r^{\delta_r}\]configurations $(q_2,\cdots,q_{k+1}) \in [d]^k$ which contribute to the previous sum; this is because there are $d^{\Delta}\prod_{r=2}^d k_r^{\delta_r}$ many choices for $(q_2,\dots,q_k)$ and at most $d$ choices for $q_{k+1}$. 
Hence,

\begin{align*}
&dk \sum_{q_2,\dots,q_{k+1}\in [d]}\sum_{\ell_2,\dots,\ell_{k}\in [d]}
\frac{d^\Delta}{n^{\Delta}}  \1\big\{\delta_1 =1\big\} \prod_{r=2}^d \1\big\{\delta_r \in\{0,1\}\big\} \\
&\leq d^2k\sum_{k_1,\dots,k_d\geq 0}\binom{k-1}{k_1,\dots,k_d}\sum_{\delta_2,\dots,\delta_d\in \{0,1\}} \prod_{r=2}^d \left(\frac{k_rd^2}{n}
\right)^{\delta_r}\\
&\le k d^{k+1}\big( 1+ \tfrac{d}{n} \partial_x \big)^d x^{k-1}\big|_{x=1}.
\end{align*}

The last inequality follows as in the proof of Lemma~\ref{lem:1} (replacing $k$ by $k-1$). 
Summing over all $k\in \mathbb N$, we conclude that for any $0<r<1$, 
\begin{align}
&\sum_{k\geq 1} \frac{r^k}{d^{k+1}}
\bigg|  \sum_{|I|=|J|=k, |I\cap J|=k-1} \E\det A(I)A(J)\bigg|  \notag\\
&  \le  \sum_{k\geq 1} 	kr^k\big( 1+ \tfrac{d}{n} \partial_x \big)^d x^{k-1}\bigg|_{x=1}
 = \big( 1+ \tfrac{d}{n} \partial_x \big)^d\frac{r}{(1-rx)^2}\bigg|_{x=1} \notag\\
& =  \sum_{\ell =0}^d \binom{d}{\ell} \frac{d^\ell}{n^\ell} \partial_x^\ell \frac{r}{(1-xr)^2} \bigg|_{x=1}  \notag\\
&= r \sum_{\ell=0}^d \frac{d!}{(d-\ell)!\ell!} \frac{(rd)^{\ell}}{n^\ell} \frac{(\ell+1)!}{(1-r)^{\ell+2}} \notag \\
& \le\frac{r}{(1-r)^2} \sum_{\ell=0}^{\infty}  (\ell+1) \bigg( \frac{d^2r}{n(1-r)} \bigg)^\ell = \frac{r}{(1-r - \frac{rd^2}{n})^2}
\label{eq:geosum} .
\end{align}
This completes the proof.
\end{proof}

\begin{proof}[Proof of Proposition \ref{prop:tightness}]
From Lemma \ref{lem:1} and Lemma \ref{lem:2}, together with \eqref{reduction}, we obtain when $1\leq d\leq \sqrt{\frac{n(1-r)}{r}}$,
\begin{align}
\sum_{k\le n} \frac{r^k}{d^{k+1}} \E  |\Delta_k^{(n)}|^2
 \le  \frac{\frac{2}{d}+r }{(1-r - \frac{rd^2}{n})^2}. 	
\end{align}
Then from \eqref{eq:unif_moment_bound}, \eqref{eq:1+eps_moment}, the conclusion holds.
\end{proof}

\subsection{Asymptotics of traces of sums of $d$ independent uniform permutations}

We now turn to the identification of the limits of $\mathrm{trace}(A^k)$ stated in Theorem \ref{thm:trace}.  {We will first need some probability estimates for the $d=1$ case that will be used in the sequel.}

\subsection{Probability estimates for the $d=1$ case}

Recall that in case $d=1$, $A = P =  (\1_{\pi(i)=j})_{i,j\in[n]}$ where $\pi$ is a uniform random element of $S_{[n]}$. 
We start by giving an estimate on the probability of generic events.
Recall that for $s\in\N$, we denote 
\[\mathcal{E}_s = \{\mathbf{i} \in [n]^s :  i_1,\dots, i_s \text{ are distinct} \}.\]

\begin{prop}\label{prop:event}
Let $r,s \in\N_0$ with $r+s < n$ and fix $\mathbf i , \mathbf j\in \mathcal{E}_s$ such that $(i_1,j_1) , \cdots,   {(i_s,j_s)}$ are distinct, and  $\mathbf k , \mathbf q \in [n]^r$. Assume
\begin{equation} \label{conddisjoint}
\big\{ (k_1,q_1) , \cdots,  (k_r,q_r) \big\} \cap  \big\{ (i_1,j_1) , \cdots,  (i_s,j_s) \big\} =\emptyset.
\end{equation}
Then there exists  $0\leq \varepsilon \leqslant \frac{r}{n-s}$ such that 
\begin{equation}\label{eq:proba}
\P\big(P_{i_1, j_1}=\dotsc= P_{i_s, j_s}=1 , P_{k_1,q_1}=\dotsc= P_{k_r,q_r}=0 \big)
= (1-\varepsilon)\frac{(n-s)!}{n!}   .
\end{equation}
Otherwise,  if  $\mathbf i \notin \mathcal{E}_s$, or  $\mathbf j \notin \mathcal{E}_s$, or the condition \eqref{conddisjoint} fails, then the probability on the LHS of \eqref{eq:proba} equals $0$.
\end{prop}

\begin{proof}
Under the hypothesis \eqref{conddisjoint},  we have the upper-bound 
\[
\P\big(P_{i_1, j_1}=\dotsc= P_{i_s, j_s}=1 , P_{k_1,q_1}=\dotsc= P_{k_r,q_r}=0 \big)
 \le  \mathbb{P}\begin{pmatrix}\pi(i_1)=j_1, \dotsc, \pi(i_s)=j_s\end{pmatrix} 
=\frac{(n-s)!}{n!}  
\]
valid for any $\mathbf i , \mathbf j\in \mathcal{E}_s$. 
For the lower-bound, we may assume that $k_1,\dots , k_r \notin \mathbf{i}$
and $q_1,\dots, q_r  \notin \mathbf{j}$. Otherwise, there are less constraints and the probability of the event in question is larger. 
In this case, using the invariance \eqref{exchange} of the uniform measure on $S_{[n]}$, we can write
\[
\mathbb{P}\begin{pmatrix}P_{i_1, j_1}=\dotsc= P_{i_s, j_s}=1 \\ 
 P_{k_1,q_1}=\dotsc= P_{k_r,q_r}=0 
\end{pmatrix} 
= \frac{(n-s)!}{n!}    \mathbf{Q}\begin{pmatrix}\pi(1)\neq 1, \dotsc, \pi(r) \neq r\end{pmatrix} 
\]
where $  \mathbf{Q}$ is the uniform measure on $S_{[n-s]}$.
In particular, we have the simple lower-bound 
\[
 \mathbf{Q}\begin{pmatrix}\pi(1)\neq 1, \dotsc, \pi(r) \neq r\end{pmatrix} 
 \ge 1- r  \mathbf{Q}\big(\pi(1) = 1\big) = 1- \frac{r}{n-s} .
\]
This proves the claim. 
\end{proof}

\subsubsection{Subgraph probability estimation}

Given integers
$0 \leqslant k_1, \dotsc, k_r \leqslant d$ with $k=k_1+\dotsb+k_r$, define
\[
\mathscr{T}_{\mathbf k} = \big\{ T \in \{0,1\}^{d \times r} : {\textstyle \sum_{i=1}^d} T_{ij} = k_j  \text{ for all } j\in[r] \big\} .
\]
Observe that 
\begin{equation} \label{boundTk}
|\mathscr{T}_{\mathbf k} | = \binom{d}{k_1} \cdots \binom{d}{k_r} \le \frac{d^k}{k_1!\cdots k_r!} .
\end{equation}

\begin{prop}\label{prop:upperbound} 
Fix $\mathbf i , \mathbf j \in [n]^r$ for $r\in\N$ such that 
$(i_1,j_1) , \cdots,  (i_r,j_r)$ are distinct. 
Let $\mathbf k \in [d]^r$ and set $k=k_1+\dotsb+k_r$. Then,  
\begin{equation}
\mathbb{P}(A_{i_1, j_1}=k_1, \dotsc, A_{i_r, j_r}=k_r) \leqslant |\mathscr{T}_{\mathbf k} |  \frac{(n-k)!}{n!}. \notag 
\end{equation}
\end{prop}

\begin{proof}
Decomposing the event $\{A_{i_1, j_1}=k_1, \dotsc, A_{i_r, j_r}=k_r\}$ in terms of the permutations $P^1, \dotsc, P^d$ leads to
\begin{equation*} \mathbb{P}(A_{i_1, j_1}=k_1, \dotsc, A_{i_r, j_r}=k_r) = \sum_{T \in \mathscr{T}} \mathbb{P}\left(\begin{bmatrix}
P^1_{i_1, j_1} & \dots & P^1_{i_r, j_r} \\
\vdots & & \vdots \\
P^d_{i_1, j_1} & \dots & P^d_{i_r, j_r}
\end{bmatrix}=T \right).
\end{equation*}
Since $P^1, \dotsc, P^d$ are independent, we obtain
\begin{equation*}\label{p:eq:prod}
\mathbb{P}(A_{i_1, j_1}=k_1, \dotsc, A_{i_r, j_r}=k_r) = \sum_{T \in \mathscr{T}}
\prod_{\delta=1}^d \mathbb{P}(P_{i_1, j_1}=T_{\delta, 1}, \dotsc, P_{i_r, j_r} = T_{\delta, r}).
\end{equation*}
Each term in this product falls under Proposition \ref{prop:event} -- either it equals $0$ or \eqref{eq:proba} holds. 
Let $s_\delta=T_{\delta,1}+\dotsb+T_{\delta, r}$ for $\delta\in [d]$,
this implies that there exists  $0\leq \varepsilon_\delta \leqslant \frac{r}{n-r}$ for $\delta\in [d]$ so that
\begin{equation}\label{p:eq:prod}
\mathbb{P}(A_{i_1, j_1}=k_1, \dotsc, A_{i_r, j_r}=k_r) \le \sum_{T \in \mathscr{T}} 
\prod_{\delta=1}^d 
(1-\varepsilon_\delta)\frac{(n-s_\delta)!}{n!} , 
\quad\text{ with equality if $\mathbf i , \mathbf j \in \mathcal E_r $.}
\end{equation}
Note that by definition of $\mathscr{T}$, $\sum_{\delta=1}^d s_\delta = \sum_{j=1}^r k_j = k$, this yields the upper-bound,
\[
\mathbb{P}(A_{i_1, j_1}=k_1, \dotsc, A_{i_r, j_r}=k_r) \le \sum_{T \in \mathscr{T}} 
\prod_{\delta=1}^d \frac{(n-s_\delta)!}{n!} 
\le |\mathscr{T}|\frac{(n-k)!}{n!} 
\]
where we used  \eqref{factorbd1}. 
\end{proof}

\begin{prop}\label{prop:eventd}
Fix $\mathbf i , \mathbf j \in \mathcal E_{r}$ for  $r\leq n/2$. Let $\theta =d/n$.
 There exists  $|\varepsilon | \le \max \left\{\frac{ d r}{n-r},\frac{2r^2}{n}\right\}$ such that 
\begin{equation}\label{eq:ewens}
\mathbb{P}\big( A_{i_1,j_1} = \cdots = A_{i_r,j_r} =1 \big)
= (1+\varepsilon) \theta^r.
\end{equation}
\end{prop}

\begin{proof}
 Note  that for $r\le n/2$, 
 \begin{align*}
 \frac{n^r(n-r)!}{n!}=\left(1+\frac{r-1}{n-r+1}\right)\cdots \left(1+\frac{1}{n-1}\right)\leq \left(1+\frac{r-1}{n-r+1}\right)^{r-1}\leq 1+\frac{2r^2}{n}.
 \end{align*}
We have the inequality for $r\leq n/2$,
\begin{align}\label{eq:factorial_bound}
	\frac{(n-r)!}{n!}  \le \frac{1+2 r^2/n}{n^r}.
\end{align}
By Proposition~\ref{prop:upperbound} and using the fact that $|\mathscr{T}_{\mathbf k} |=d^r$ for $\mathbf k= (1,\dots,1)$, and \eqref{eq:factorial_bound}, we get 
\[ \mathbb{P}\big( A_{i_1,j_1} = \cdots = A_{i_r,j_r} =1 \big)
\leq \left(1+\frac{2r^2}{n}\right)  \theta^r.\]
For the lower bound, by formula \eqref{p:eq:prod} in case $\mathbf i , \mathbf j \in \mathcal E_r$, 
\[
\mathbb{P}(A_{i_1, j_1}=k_1, \dotsc, A_{i_r, j_r}=k_r) \ge (1-\varepsilon)^d  \sum_{T \in \mathscr{T}} 
\prod_{\delta=1}^d 
\frac{(n-s_\delta)!}{n!} 
\]
with $\varepsilon =\frac{r}{n-r}$. 
By convexity, when $r<n/2$,  $(1-\varepsilon)^d  \ge 1- d\epsilon $. Then, using that in this case $\sum_{\delta=1}^d s_\delta = \sum_{j=1}^r k_j = r$, this implies that 
\[
\mathbb{P}(A_{i_1, j_1}=k_1, \dotsc, A_{i_r, j_r}=k_r) \ge (1-d\varepsilon)  | \mathscr{T}| n^{-r} .
\]
This concludes the proof.
\end{proof}

As a direct consequence of Proposition~\ref{prop:eventd}, we have the asymptotics of the probability of a simple cycle on the digraph of the random matrix $A$. 
We now turn to bounding  joint moments of the entries of the random matrix $A$.

\begin{prop} \label{prop:mom}
For a fixed constant $r\in\N$, fix $\boldsymbol \alpha , \boldsymbol \beta \in \N^r$ and let $\beta = \beta_1+\cdots +\beta_r$. Let $\theta =d/n$ and assume that $ d \le \sqrt n$. Then, for any $\mathbf i , \mathbf j \in  [n]^r$, 
\[
\mathbb{E}\big[A_{i_1, j_1}^{\alpha_1} \dotsb  A_{i_r, j_r}^{\alpha_r} \1_{\{A_{i_\ell, j_\ell} \ge \beta_\ell, \ell \in [r]\}}\big]  = O_{\boldsymbol \alpha,\boldsymbol \beta}(\theta^\beta) .
\]
\end{prop}

\begin{proof}
Without loss of generality, we may assume that $(i_1,j_1) , \cdots,  (i_r,j_r)$ are distinct. 
By Proposition~\ref{prop:upperbound} and \eqref{boundTk}, we have
\[\begin{aligned}
\mathbb{E}\big[A_{i_1, j_1}^{\alpha_1} \dotsb  A_{i_r, j_r}^{\alpha_r} \1_{\{A_{i_\ell, j_\ell} \ge \beta_\ell, \ell \in [r]\}}\big]  
& = \sum_{\mathbf k\in[d]^r, \mathbf k \ge \boldsymbol \beta } k_1^{\alpha_1}\cdots  k_r^{\alpha_r} \mathbb{P}(A_{i_1, j_1}=k_1, \dotsc, A_{i_r, j_r}=k_r)  \\
&\le  \sum_{\mathbf k\in[d]^r, \mathbf k \ge \boldsymbol \beta }  \frac{k_1^{\alpha_1}\cdots  k_r^{\alpha_r}}{k_1!\cdots k_r!} d^{k}  \frac{(n-k)!}{n!} ; \qquad k=k_1+\cdots +k_r .
\end{aligned}\]
Again with \eqref{eq:factorial_bound}, we can bound $\frac{(n-k)!}{n!}  \le \frac{1+2 k^2/n}{n^k}$ for $k\le n/2$.
Since $r$ is fixed, this yields the estimate valid for $k\leq  rd \le n/2$, and  $d\leq \sqrt{n}$,
\[\begin{aligned}
\mathbb{E}\big[A_{i_1, j_1}^{\alpha_1} \dotsb  A_{i_r, j_r}^{\alpha_r}\1_{\{A_{i_\ell, j_\ell} \ge \beta_\ell, \ell \in [r]\}}\big] 
& \le (1+2r^2d^2/n)\sum_{\mathbf k  \ge \boldsymbol \beta} \prod_{j=1}^r  \frac{k_j^{\alpha_j}}{k_j!}  \theta^{k_j} \leq (1+2r^2)\sum_{\mathbf k  \ge \boldsymbol \beta} \prod_{j=1}^r  \frac{k_j^{\alpha_j}}{k_j!}  \theta^{k_j}.
\end{aligned}\]
Now, we can use the bound for $\theta \le 1$,
\[
\sum_{k\ge b} \frac{k^a}{k!} \theta^k \le C_{a,b} \theta^b ,
\]
 to conclude the proof.
\end{proof}

\subsubsection{Reducing traces to cycle counts in the digraph of $A$}

Recall that
\[
\mathrm{tr}(A^k) = \sum_{\mathbf{i} \in [n]^k}A_{\mathbf{i}}.
\]
In particular, by Proposition~\ref{prop:mom} applied with $\alpha_1+\cdots+\alpha_r=k$, $\beta_1=\cdots=\beta_r=1$, 
\[
\E \mathrm{tr}(A^k)  \le C_k n^k \theta^k = C_k d^k . 
\]
In particular, in the regime where $d$ is fixed, for any $k\in\N$, the non-negative random variables $\big(\mathrm{tr}(A^k)\big)_{n\in\N}$ are tight.

 For every $\mathbf{i} = (i_1, \dotsc, i_k) \in [n]^k$, we can associate a digraph 
\begin{align*}\label{eq:v-e}
&V(\mathbf{i}) = \{i_1, \dotsc, i_k\} && v(\mathbf{i}) = \#V(\mathbf{i}) \\
&E(\mathbf{i}) = \{(i_1, i_2), (i_2, i_3), \dotsc, (i_{k-1}, i_k), (i_k, i_1) \} &&e(\mathbf{i}) = \#E(\mathbf{i}).
\end{align*}
One can interpret $v(\mathbf{i})$ as the number of `vertices' of $[n]$ that are visited by $\mathbf{i}$, and $e(\mathbf{i})$ as the number of distinct `edges' of $\mathbf{i}$. The  digraph $(V(\mathbf{i}), E(\mathbf{i}))$ contains at least a loop, so necessarily $v(\mathbf{i})\leqslant e(\mathbf{i}) \le k$. 
 For integers $v \leqslant e \leqslant k$, we define
\begin{equation}
\mathscr{E}_k(v,e) = \{\mathbf{i} \in [n]^k : v(\mathbf{i})= v, e(\mathbf{i})=e\}.
\end{equation}

Let us decompose
\begin{equation}\label{Rk}
\tr(A^k) = \sum_{\substack{\mathbf{i} \in \mathscr{E}_k(v,e) \\ v=e}}\mathbf{1}_{A_\mathbf{i}=1} + \sum_{\substack{\mathbf{i} \in \mathscr{E}_k(v,e) \\ v=e}}A_\mathbf{i}\mathbf{1}_{A_\mathbf{i}>1}+ \sum_{\substack{\mathbf{i} \in \mathscr{E}_k(v,e) \\ v<e}}A_\mathbf{i} =: T_k+R_k+S_k.
\end{equation}

First, let us show that both $R_k$ and $S_k$ can be treated as negligible errors.

\begin{lemma} \label{lem:Rk}
Suppose that $d = n^{o(1)}$ as  $n \to \infty$.
For any fixed $k\in \N$, $(R_k+S_k)\to0$ in probability   as $n \to \infty$.
\end{lemma}

\begin{proof}
For any $\mathbf{i} \in [n]^k$, by Proposition~\ref{prop:mom}  with $k$ fixed,
\[
\E A_\mathbf{i} \le C_k \theta^{e(\mathbf i)} , \qquad \theta = d/n,
\]
since there are $e$ different entries and we can take $\boldsymbol \beta  = \underbrace{(1,\dots,1)}_{\times e}$. 
Since $|\mathscr{E}_k(v,e)| \le k! n^{v} $, this shows that 
\[
\E S_k \le \sum_{v<e\le k} C_k n^v \theta^e \le C_k \frac{d^k}{n} .
\]
We conclude that if  $d = n^{o(1)}$, then $\E S_k \to 0$  as  $n \to \infty$.
Similarly, by Proposition~\ref{prop:mom},
\[
\E A_\mathbf{i}\mathbf{1}_{A_\mathbf{i}>1} \le C_k \theta^{e(\mathbf i)+1} ,
\]
so that 
\[
\E R_k \le \sum_{v\le k} C_k n^v \theta^{v+1} \le C_k \frac{d^{k+1}}{n} .
\]
Hence, also  $\E R_k \to 0$  as  $n \to \infty$. 
\end{proof}

\begin{lemma} [Lemma 9.3 in \cite{coste2021sparse}]\label{lem:Coste}
$\mathscr{E}_{k}(v,v) $ is empty if $v$ is not a divisor of $k$. Otherwise, if $k = v q$, then the elements of $\mathscr{E}_{k}(v,v)$ are exactly the sequences 
\begin{equation}
( i_1, i_2, \dotsc, i_v, i_1, \dotsc, i_v, \dotsc, i_1, \dotsc, i_v)
\end{equation}
where the subsequence $\mathbf{i}'= (i_1, \dotsc, i_v) \in \mathcal{E}_{v} $ is repeated $q$ times.
Moreover, the events $\{A_{\mathbf{i}} = 1 \} = \{A_{\mathbf{i}'} = 1\}$.
\end{lemma}

Lemma~\ref{lem:Coste} implies that according to \eqref{Rk}, 
\[
T_k  = \sum_{v | k} \sum_{\mathbf{i}' \in \mathscr{E}_{v}}\1_{A_{\mathbf{i}'} = 1}.
\]

Let us denote for $\ell\in[n]$,
\begin{equation} \label{cycles}
\mathcal C_\ell : = \big\{ \mathbf{i}=(i_1, \dotsc, i_\ell) \text{ modulo cyclic permutation}  :  i_k \in [n] , \text{ distincts}\big\} 
\end{equation}
and the random variables
\begin{equation} \label{Qvar}
Q_\ell := \sum_{\mathbf{i} \in \mathcal{C}_{\ell}}\1_{A_{\mathbf{i}}=1} .
\end{equation}
With this new notation,
\begin{equation} \label{Tk}
T_k  = \sum_{\ell | k} \ell Q_\ell . 
\end{equation}
The interpretation is that $ Q_\ell $ is the number of (oriented) $\ell$-cycles on the digraph defined by the random matrix $A$ and $T_k$ is a good approximation for $\mathrm{tr}(A^k)$.

\subsubsection{Asymptotics of joint moments of cycle counts}

If $Q\in\N$, we denote its falling factorials by 
\[
(Q)_r = Q(Q-1)\cdots (Q-r+1) , \qquad r\in\N .
\]
Our interest in these quantities stems from the fact that if $Q = |\mathcal{C}|$ for a finite set $\mathcal{C}$, then 
\begin{equation}
\label{combQ}
(Q)_r = \big| \{ (j_1, \dotsc, j_r)  : j_k \in \mathcal{C} \text{ distincts}\}\big| . 
\end{equation}

The  following basic probabilistic result holds. 
\begin{lemma} \label{lem:Poi}
For $\lambda>0$, $\Lambda \sim {\rm Poisson}(\lambda)$ if and only if for all $r\in\N$,
\[
\E (\Lambda)_r = \lambda^r .
\]
\end{lemma}

\begin{prop}\label{thm:poisson}
Recall the notation \eqref{Qvar} and fix $k\in\N$ and $\boldsymbol \alpha \in \N_0^k$. Then,  as $n\to\infty$ 
\begin{equation*}
\mathbb{E}\left[ (Q_{1})_{\alpha_1} \cdots (Q_{k})_{\alpha_k} \right]= { \left( \frac{d^{1}}{1}\right)^{\alpha_1}  \dotsb  \left( \frac{d^{m}}{m}\right)^{\alpha_m}}  \big(1+ O_{\boldsymbol \alpha}(d/n) \big).
\end{equation*}
\end{prop}

The proof of Proposition~\ref{thm:poisson} will be given in Section~\ref{sec:proofmom}. It relies on Proposition~\ref{prop:eventd} and basic combinatorial arguments counting certain collections of (distinct) cycles on the digraph of the random matrix $A$.  
To illustrate the argument, we compute the first and second moment of the random variable $Q_\ell$ in Section~\ref{sec:proofmom2}.

According to Lemma~\ref{lem:Poi}, Proposition~\ref{thm:poisson} directly implies that for a fixed $d\in\N$, for any fixed $\boldsymbol\ell , \boldsymbol \alpha \in \N^k$,
\begin{equation}\label{cv_vers_poisson}
\lim_{n\to\infty}\mathbb{E}\left[ (Q_{\ell_1})_{\alpha_1} \cdots (Q_{\ell_k})_{\alpha_k} \right]
=\mathbb{E}\left[ (\Lambda_{\ell_1})_{\alpha_1} \cdots (\Lambda_{\ell_k})_{\alpha_k} \right]
\end{equation}
where the Poisson random variables $\{\Lambda_\ell \}_ {\ell \in \mathbb{N}}$  are given by Definition~\ref{def:1}.  
Now we are ready to prove Theorem \ref{thm:trace}.

\subsubsection*{Proof of \eqref{eq:fix_convergence}}
 The joint convergence of the factorial moments described in \eqref{cv_vers_poisson} implies that for every $k\in\N$,
 \[
 (Q_{1}, \dotsc, Q_{k}) \xrightarrow[n \to \infty]{\mathrm{law}}\ (\Lambda_1, \dotsc, \Lambda_k)
 \]
  and in the sense of moments. 
  Then, by the Cram\'er-Wold theorem, for every $k\in\N$,
\[
\bigg( T_1 = Q_1 , \cdots , T_k =  \sum_{\ell| k}\ell Q_\ell \bigg)  \xrightarrow[n \to \infty]{\mathrm{law}} \bigg(\Lambda_1, \dotsc, \sum_{\ell | k} \ell \Lambda_\ell \bigg).
\]
 By Lemma~\ref{lem:Rk}, for every $k\in\N$ and any $\epsilon>0$,
\begin{equation} \label{traceproba}
 \P\big(  |\tr(A)-T_1| + \cdots + |\tr(A^k)-T_k| \ge \epsilon \big) 
 \xrightarrow[n \to \infty]{} 0
 \end{equation}
 This establishes that for any fixed $d\in\N$ and $k\in\N$,
\[
\big( \tr(A), \cdots , \tr(A^k) \big)  \xrightarrow[n \to \infty]{\mathrm{law}} \bigg(\Lambda_1, \dotsc, \sum_{\ell | k} \ell \Lambda_\ell \bigg).
\]

\subsubsection*{Proof of \eqref{eq:grow_convergence}}

In the regime where the degree $d=d(n)\to\infty$ so that $d=n^{o(1)}$ as $n\to\infty$, the error term in Proposition~\ref{thm:poisson} is so good that we can directly show that the random variables $(Q_\ell)_{\ell\in\N}$, if suitably normalized, converge to independent Gaussians in the sense of finite dimensional distributions. 
Namely, the asymptotics from Proposition~\ref{thm:poisson} yields the conditions \eqref{Poimom}--\eqref{PoiCLTcond} in Lemma \ref{lem:PoiCLT} with 
\[
\lambda_{i,n} = d(n)^i/i \qquad\text{and}\qquad
\epsilon_{\mathbf k,n} = O_{\mathbf k}(d(n)/n) .
\]
In particular since for every $k\in\N$, $d(n)^k/n \to0$ as $n\to\infty$,
by Lemma~\ref{lem:PoiCLT}, we conclude that for any $\ell\in\N$
\begin{equation} \label{QmultiCLT}
\left(\frac{Q_{1}-d}{\sqrt{d/1}}, \dots, \frac{Q_{\ell}- d^\ell/\ell}{\sqrt{d^\ell/\ell}} \right)  \xrightarrow[n \to \infty]{\mathrm{law}} \big( N_1,\dots, N_\ell \big) ,
\end{equation}
as well as in the sense of moments, where  $\{N_\ell\}_{\ell \in \mathbb N}$  are as in Definition~\ref{def:2}.
Observe that according to \eqref{Tk}, we can write for $k\in\N$,
\begin{equation} \label{TQk}
\frac{T_k - d^k}{d^{k/2}}  = \sum_{\ell | k} \sqrt\ell d^{-(k-\ell)/2} \bigg(  \frac{Q_{\ell}- d^\ell/\ell}{\sqrt{d^\ell/\ell}}  \bigg)+  \sum_{\ell | k , \ell <k} d^{\ell-k/2} 
\end{equation}
Hence, by \eqref{QmultiCLT} and the Cram\'er-Wold theorem, for every $k\in\N$,
\[
\left(\frac{T_{1}-d}{\sqrt{d}}, \dots, \frac{T_{2k}-d^{2k}}{\sqrt{d^{2k}}} \right)  \xrightarrow[n \to \infty]{\mathrm{law}} \big(N_1, \sqrt 2 N_2 +1,  \sqrt 3 N_3, \dots, \sqrt{2k} N_{2k} + 1\big) .
\]
Note that in \eqref{TQk}, as $d(n)\to\infty$, only the term $\ell = k$ contributes to the random and the term $\ell = k/2$ contributes to the mean if $k\in 2\N$.
In particular, once normalized, the weak limit of $T_k$ are still independent Gaussians.

To finish the proof of \eqref{eq:grow_convergence}, by \eqref{Rk} and   Lemma~\ref{lem:Rk} which  still holds in the regime where $d=n^{o(1)}$, for a fixed $k\in\N$,
\[
\left(\frac{\tr(A^{1})-d}{\sqrt{d}}, \dots, \frac{\tr(A^{2k})-d^{2k}}{\sqrt{d^{2k}}} \right)  
=
\left(\frac{T_{1}-d}{\sqrt{d}}, \dots, \frac{T_{2k}-d^{2k}}{\sqrt{d^{2k}}} \right)   + \underset{n\to\infty}{o(1)} 
\]
where the error is controlled as in  \eqref{traceproba}.
This completes the proof of Theorem~\ref{thm:trace}.

\subsection{Proof of Proposition \ref{thm:poisson}}

We now prove Proposition \ref{thm:poisson}. We first prove the proposition only for first and second moments, to give a flavour of the proof. The complete proof is in Subsection \ref{sec:prop4.3}. 

\subsubsection{First and second moment}
\label{sec:proofmom2}

Let us first prove that 
\begin{equation}
  \mathbb{E}[Q_i] = \frac{d^i}{i} (1+O_i(d/n)),\notag 
\end{equation}
and that
\begin{equation}
  \mathbb{E}[(Q_i)_2] = \left(\frac{d^i}{i}\right)^2 (1 + O_i(d/n)). \notag 
\end{equation}
By \eqref{Qvar} and since $|\mathcal C _\ell| = (n)_\ell /\ell$, according to Proposition~\ref{prop:eventd}, there exists  $ \varepsilon = O_\ell(d/n)$ so that
\[
\E Q_\ell =  |\mathcal C _\ell| (1+\varepsilon) (d/n)^\ell . 
\]
For fixed $\ell \in\N$, this shows that as $n\to\infty$,
\[
\E Q_\ell = d^\ell/\ell \big(1+ O_\ell(d/n) \big).
\]

For the second (factorial) moment, by \eqref{combQ},
\[
\E (Q_\ell)_2 =  \sum_{\mathbf i , \mathbf j \in\mathcal C _\ell , \mathbf i \neq \mathbf j } \P\big( A_\mathbf{i}=1 ,  A_\mathbf{j}=1 \big) .
\]
If we identify $\mathbf i , \mathbf j$ with $\ell$-cycles, observe that the condition  $\mathbf i \neq \mathbf j$ implies that the digraph $\{\mathbf i , \mathbf j\}$ obtained by concatenating  $\mathbf i$, $\mathbf j$, 
\begin{equation} \label{digij}
  \begin{aligned}
&V(\{\mathbf{i},\mathbf{j}\}) = \{i_1, \dotsc, i_\ell, j_1, \dotsc, j_\ell\}  \\
&E(\{\mathbf{i},\mathbf{j}\}) = \{(i_1, i_2), (i_2, i_3), \dotsc, (i_{k-1}, i_\ell), (i_\ell, i_1) , (j_1, j_2), (j_2, j_3), \dotsc, (j_{k-1}, j_\ell), (j_\ell, j_1) \} 
\end{aligned}
\end{equation}
satisfies 
\begin{equation} \label{disjij}
\mathbf{i} \cap \mathbf{j}  = \emptyset \quad\text{if and only if}\quad
|V(\{\mathbf{i},\mathbf{j}\})| = |E(\{\mathbf{i},\mathbf{j}\})| .
\end{equation}
In particular, we can split 
\[
\E (Q_\ell)_2 =  \sum_{\mathbf i , \mathbf j \in\mathcal C _\ell , \mathbf{i} \cap \mathbf{j}  = \emptyset} \P\big( A_\mathbf{i}=1 ,  A_\mathbf{j}=1 \big) 
+   \sum_{\mathbf i , \mathbf j \in\mathcal C _\ell , |V|<|E|} \P\big( A_\mathbf{i}=1 ,  A_\mathbf{j}=1 \big) 
\]
where the second term is due  to \eqref{disjij}. By Proposition~\ref{prop:upperbound}, we have for $\mathbf i , \mathbf j \in\mathcal C _\ell$,
\[
\P\big( A_\mathbf{i}=1 ,  A_\mathbf{j}=1 \big)  \le C_\ell (d/n)^{|E|}
\]
and $\big| \{\mathbf i , \mathbf j \in\mathcal C _\ell , |V|=v\} \big| \le C_\ell n^v$ for $v\in \{\ell , \cdots , 2\ell\}$, so that 
\[
  \sum_{\mathbf i , \mathbf j \in\mathcal C _\ell , |V|<|E|} \P\big( A_\mathbf{i}=1 ,  A_\mathbf{j}=1 \big)  \le C_\ell \frac{d^{2\ell}}{n} .
\]

On the other hand, also by Proposition~\ref{prop:eventd}, for $\mathbf i , \mathbf j \in\mathcal C _\ell$  disjoint, there exists  $ \varepsilon = O_\ell(d/n)$ so that
\[
 \P\big( A_\mathbf{i}=1 ,  A_\mathbf{j}=1 \big) =  (1+\varepsilon) (d/n)^{2\ell} .
\]
Moreover, for a fixed $\ell\in\N$, we have
$\big|\{\mathbf i , \mathbf j \in\mathcal C _\ell , \mathbf{i} \cap \mathbf{j}  = \emptyset\} \big| = \tfrac{n^{2\ell}}{\ell^2}\big(1- O_\ell(1/n)\big)$ as $n\to\infty$. 
This shows that  as $n\to\infty$
\begin{equation} \label{momQ2}
\E (Q_\ell)_2 =  \frac{d^{2\ell}}{\ell^2}  \big(1+ O_\ell(d/n) \big).
\end{equation}

\subsubsection{Proof of Proposition \ref{thm:poisson}}\label{sec:prop4.3}
\label{sec:proofmom}

Fix $k\in\N$ and $ \boldsymbol \alpha \in \N_0^k$. 
Recall \eqref{cycles} and let us denote 
\[
\Gamma_{\boldsymbol \alpha } = \big\{ \vec{\mathbf i} =
\{ \mathbf{i}_{r,\ell} \}_{r\in [\alpha_\ell] , \ell \in [k]} : \, \mathbf{i}_{1,\ell} , \dots , \mathbf{i}_{\alpha_\ell,\ell} \in \mathcal C _{\ell}, \text{ distincts, for all }\ell\in[k]
\big\} .
\]
An element $\vec{\mathbf i} \in \Gamma_{\boldsymbol \alpha}$ is a collection of distinct cycles of length $(\underbrace{1,\cdots , 1}_{\times \alpha_1}, \cdots,\underbrace{k,\cdots , k}_{\times \alpha_k})$. 
Then, by \eqref{combQ}, we have 
\[
\mathbb{E}\left[ (Q_{\ell_1})_{\alpha_1} \cdots (Q_{\ell_k})_{\alpha_k} \right]
=  \sum_{ \vec{\mathbf i}\in\Gamma_{\boldsymbol \alpha}} \P\big( A_{\mathbf{i}_{r,\ell}}=1 , \forall r\in [\alpha_\ell] , \forall \ell \in [k] \big) .
\]
Generalizing the notation \eqref{digij}, the digraph $\vec{\mathbf i}$ satisfies 
\begin{equation*} \label{disji}
|V(\vec{\mathbf i})| = |E(\vec{\mathbf i})|
\quad\text{if and only if}\quad 
\text{the cycles $\mathbf{i}_{r,\ell}$ are all disjoint.}
\end{equation*}
This follows e.g.~by induction on \eqref{disjij}. 
Hence, we split
 \begin{align} \label{jm1}
\mathbb{E}\left[ (Q_{\ell_1})_{\alpha_1} \cdots (Q_{\ell_k})_{\alpha_k} \right]
 &=  \sum_{ \vec{\mathbf i}\in\Gamma_{\boldsymbol \alpha} , \, \mathbf{i}_{r,\ell} \text{ disjoint} } \hspace{-.5cm} \P\big( A_{\mathbf{i}_{r,\ell}}=1 , \forall r\in [\alpha_\ell] , \forall s \in [k] \big)  \\
 &\label{jm2}
 +  \sum_{ \vec{\mathbf i}\in\Gamma_{\boldsymbol \alpha} , \, |V(\vec{\mathbf i})| < |E(\vec{\mathbf i})|} \hspace{-.5cm} \P\big( A_{\mathbf{i}_{r,\ell}}=1 , \forall r\in [\alpha_\ell] , \forall s \in [k] \big) .
\end{align}

By Proposition~\ref{prop:eventd}, for any $\vec{\mathbf i}\in\Gamma_{\boldsymbol \alpha}$, there exists  $ \varepsilon(\vec{\mathbf i}) = O(d/n)$ so that 
\[
\P\big( A_{\mathbf{i}_{r,\ell}}=1 , \forall r\in [\alpha_\ell] , \forall \ell \in [k] \big) =
 (1+\varepsilon(\vec{\mathbf i})) (d/n)^{E(\vec{\mathbf i})} .
\]
Since $\big|\{\vec{\mathbf i}\in\Gamma_{\boldsymbol \alpha} , |V|=v\} \big|=O(n^v)$ for $v\in \{1 , \cdots ,  |\boldsymbol \alpha|\}$ where $ |\boldsymbol \alpha| = 1\cdot \alpha_1 +\cdots + k\cdot \alpha_k$, we obtain
\[
\eqref{jm2} = O\big(d^{ |\boldsymbol \alpha|}/n\big) . 
\]
Now, if  $\mathbf{i}_{r,\ell}$ are all disjoint, $|V(\vec{\mathbf i})| = |E(\vec{\mathbf i})| = |\boldsymbol \alpha|$ and  \[
\big| \{ \vec{\mathbf i}\in\Gamma_{\boldsymbol \alpha} , \, \mathbf{i}_{r,\ell} \text{ disjoint} \} \big|  =\frac{n^{ |\boldsymbol \alpha|}}{1^{\alpha_1} \cdots k^{\alpha_k}} \big(1-O(1/n)\big) 
\]
This implies that 
\[
\eqref{jm1} = \frac{d^{ |\boldsymbol \alpha|}}{1^{\alpha_1} \cdots k^{\alpha_k}}   \big(1+ O(d/n) \big).
\]
Note that the implied constants above depend only on $\boldsymbol \alpha$.
This completes the proof of Proposition \ref{thm:poisson}.

\section{Poisson analytic functions} \label{sec:PAF}
\begin{proof}[Proof of Proposition~\ref{prop:PAF}]
The function $Y_d$ is centred and it is almost surely analytic in $D_{1}$ since $   \|Y_d\|^2_{L^\infty(D_{r})} <\infty$ almost surely for any $r<1$. 
Indeed, by Jensen's inequality, it holds  for $r<1$,  
\[
\E \|Y_d\|_{L^\infty(D_{r})}^{2} \le \E\bigg( \sum_{k\in\N} \frac{r^k}{k d^{k/2}} \bigg| \sum_{\ell | k} \ell \overline{\Lambda_\ell} \bigg| \bigg)^{2}
\le \frac{r}{1-r}  \sum_{k\in\N} \frac{r^k}{k^2 d^{k}} \E \bigg| \sum_{\ell | k} \ell \overline{\Lambda_\ell} \bigg|^2.
\]
By independence  of $\{\Lambda_\ell \}_ {\ell \in \mathbb{N}}$, it holds for $k\in\N$,
\[
\frac{1}{k^2 d^k} \E \bigg| \sum_{\ell | k} \ell \overline{\Lambda_\ell} \bigg|^2
= \frac{1}{k^2 d^k}  \sum_{\ell | k} \ell^2  \E \big| \overline{\Lambda_\ell} \big|^2
= \frac{1}{k^2 d^k}  \sum_{\ell | k} \ell d^\ell \le 1  ,
\]
where we used that $\operatorname{Var}(\Lambda_\ell) = d^\ell/\ell$ for $\ell\in\N$. This shows that  for any $r<1$,
\begin{equation} \label{Xdtight}
\E \|Y_d\|_{L^\infty(D_{r})}^{2}  \le \frac{r^2}{(1-r)^2}  . 
\end{equation}
Repeating this argument clearly leads to the same estimate for the random analytic function $X_d$. This proves that $Y_d$ is well-defined. Then by rearranging the series, we can write
\begin{equation} \label{Ysum}
Y_d(z)  =  \sum_{\ell\in\N} \overline{\Lambda_\ell} \sum_{k\in\N} \frac{z^{k\ell}}{k d^{k\ell/2}}
 = -\sum_{\ell\in\N}  \overline{\Lambda_\ell} \log\big(1- (z/\sqrt{d})^\ell \big) 
\end{equation}
for the  {principal} branch of $\log(1+z)$ which is analytic for $z\in D_1$.
Recall that the Laplace transform of a random variable $\Lambda\sim \mathrm{Poisson}(\lambda)$ satisfies 
\begin{equation} \label{Poimgf}
\E \exp(z(\Lambda-\E \Lambda)) = \exp\big(\lambda (e^z-1-z)\big) , \qquad z\in\C. 
\end{equation}
In particular, this can be used to give an alternative proof that the series \eqref{Ysum} is (almost surely) absolutely convergent and to justify rearranging this sum. 
 {Indeed, for any $x\in \mathbb{R}$ we have $e^{|x|} \leqslant e^x + e^{-x}$, hence using \eqref{Poimgf} with $z=\pm t/\sqrt{\lambda}$, it holds for all $0\le t \le \sqrt{\lambda}$,
\begin{align*}
\E \exp\big(t |\Lambda-\E \Lambda| / \sqrt{\lambda} \big) &\le  \exp \left (\lambda \left(e^{t/\sqrt{\lambda}}-1-t/\sqrt{\lambda}\right)\right)+ \exp \left (\lambda \left(e^{-t/\sqrt{\lambda}}-1+t/\sqrt{\lambda}\right)\right)  \leq 2\exp\big(t^2\big), 
\end{align*}
where in the last inequality, we use the fact that $e^{x}\leq 1+x+x^2$ for $|x|\leq 1$.} 
For $d\ge 2$, by Markov's inequality, we obtain the large deviation estimate,
\[
\P\big[ |\overline{\Lambda_\ell}| \ge \sqrt{d^\ell} \big] \le \exp\big( -\ell/4 \big).
\]
Hence,  almost surely, $|\overline{\Lambda_\ell}| \le \sqrt{d^\ell}$ for all $\ell\in\N$ sufficiently large and \eqref{Ysum} is absolutely convergent. 

\begin{remark} \label{rk:d1}
In the special case $d=1$, observe that
\[
\P\big[  \Lambda_\ell  \ge 2  \text{ infinitely often}\big] \le 
\lim_{n\to\infty}\sum_{\ell \ge n}\bigg(  \P\big[  \Lambda_\ell  \ge 2 \big] \bigg)
\le \bigg( \lim_{n\to\infty}\sum_{\ell \ge n} \frac{1/2}{\ell^2} \bigg) = 0 .
\]
In particular,  the random analytic function
$\sum_{k\in\N} \frac{z^k}{k}  \sum_{\ell | k} \ell \Lambda_\ell$
converges almost surely for $z\in D_1$. 
In contrast, for $d\ge 2$, it is necessary to re-center the Poisson random coefficients to define a random analytic function in~$D_1$. 
\end{remark}

Let us denote by $\mathrm{f}_\ell(z): = (1-z^\ell)e^{z^\ell}$ for $\ell\in\N$ so that $\log \mathrm{f}_\ell$ is analytic in $D_{1}$ (using the principal branch).
By independence of $\{\Lambda_\ell \}_ {\ell \in \mathbb{N}}$ and \eqref{Poimgf}, it follows from the expansion \eqref{Ysum} that
\begin{equation*} 
\E e^{-Y_d(z)} = \bigg(  \prod_{\ell\in\N}  \mathrm{f}_\ell(z/\sqrt{d})^{\frac{d^\ell}{\ell}} \bigg)^{-1} .
\end{equation*}
This infinite product converges since $\big|\log \mathrm{f}_\ell(z) \big| \le \frac{r^{2\ell}}{1-r}$ for $z\in D_r$ and we used that 
\[
e^{\log(1-z^\ell)}-1-\log(1-z^\ell)  = - \log \mathrm{f}_\ell(z) .
\]
This proves \eqref{expmd}.
Now, if we repeat the computation leading to \eqref{Xdtight}, then for any $r < d^{1/4}$,
\begin{align}
\E \|\Upsilon_d\|_{L^\infty(D_{r})}^{2} 
= \E \bigg\| \sum_{k\in\N} \frac{ z^k}{k d^{k/2}} \sum_{\ell | k, \ell<k} \ell \overline{\Lambda_\ell} \bigg\|_{L^\infty(D_{r})}^{2}
& \le  \frac{r}{d^{1/4}-r}  \sum_{k\in\N} \frac{r^k}{k^2 d^{3k/4}} \E \bigg| \sum_{\ell | k , \ell<k} \ell \overline{\Lambda_\ell} \bigg|^2 \notag  \\
 &\le \frac{r/2}{d^{1/4}-r}  \sum_{k\in\N} \frac{r^k}{d^{k/4}}
  =    \frac{r^2/2}{(d^{1/4}-r)^2},\label{eq:Upsilon_d}
\end{align}
where we used that 
\(
\sum_{\ell | k , \ell<k } \ell d^\ell  
\le  \frac{k^2}2 d^{k/2}.
\)
This shows that the random analytic function $\Upsilon_d$ converges almost surely in the disk $D_{d^{1/4}}$, which is strictly larger than $D_1$ if the degree $d\ge 2$.
On the other-hand, by independence  of $\{\Lambda_\ell \}_ {\ell \in \mathbb{N}}$, the covariance kernel of $X_d$ is
\begin{align}\label{eq:CovXd}
\E X_d(z) X_d(w) = \sum_{k\in\N} \frac{z^kw^k}{d^{k}} \E\overline{\Lambda_k}^2 
=  \sum_{k\in\N} \frac{z^kw^k}{k}  = \log(1-zw)^{-1} , \qquad z,w\in D_{1}  . 
 \end{align}
The estimate of $\mathbb E[Y_d(z)Y_d(w)]$ now follows from \eqref{eq:CovXd} and the bound \eqref{eq:Upsilon_d}. This finishes the proof of Proposition \ref{prop:PAF}. 
\end{proof}

\begin{proof}[Proof of Proposition~\ref{prop:GAF}]
We have the following multi-dimensional CLT, for any $k\in\N$,
\begin{equation} \label{PoissCLT}
\big( \overline{\Lambda_1}/\sqrt{d} , \sqrt{2} \overline{\Lambda_2}/d , \cdots , \sqrt{k} \overline{\Lambda_k} /\sqrt{d^k} \big) \xrightarrow[d \to \infty]{\mathrm{law}} (N_1, \cdots, N_k) ;
\end{equation}
cf.~Section~\ref{sec:Poi}.
In particular, we can choose a coupling, where almost surely, 
$\sqrt{\ell} \overline{\Lambda_\ell} /\sqrt{d^\ell} \to  (-1)^{\ell+1}N_\ell$ for all $\ell \in\N$ as $d\to\infty$. 
Within this coupling, and the tightness  proved in \eqref{Xdtight} and \eqref{eq:Upsilon_d}, we verify that almost surely, $d\to\infty$
\begin{align*}
X_d \to X_\infty 
\qquad\text{and}\qquad
\Upsilon_d \to 0
\end{align*}
uniformly on compact subsets of $D_1$.
\end{proof}


\section{ Proofs of Theorem~\ref{thm:mc} and Theorem~\ref{thm:gmc}}
\label{sec:cvg}

We use the following lemma from \cite{bordenave2020convergence}, which is inspired from \cite{shirai2012limit}.

\begin{lemma}[Lemma 3.2 in \cite{bordenave2020convergence}]\label{lem:finite_convergence}
Let $\{f_n : D_1 \to \C\}$ be a sequence of random analytic functions,
$f_n(z)=\sum_{k=0}^{\infty}  a_k^{(n)} z^k$ for $n\in\N$. 
Assume that $\{f_n\}$ is tight and the process $(a_k^{(n)})_{k\in\N_0} \to (a_k)_{k\in\N_0}$ converges in finite-dimensional distributions; for every $k\geq 0$, 
\begin{equation} \big(a_0^{(n)}, \dotsc, a_k^{(n)}\big) \xrightarrow[n \to \infty]{\mathrm{law}}\big(a_0, \dotsc, a_k \big) ,
\end{equation}
then $f=\sum_{k=0}^{\infty}a_k z^k$ is convergent almost surely in $D_1 $ and 
\[    
f_n\xrightarrow[n \to \infty]{\mathrm{law}} f.
\]
\end{lemma}

We also need the following non-trivial observation.

\begin{lemma} \label{lem:mean}
According to Definition~\ref{def:1}, 
\[
\E e^{-Y_d(z)}  = \exp\bigg( \sum_{k\in\N} \frac{z^{k}}{k d^{k/2}}  \sum_{\ell | k , \ell<k} \ell \, \E \Lambda_\ell \bigg)
\]
where the sum converges absolutely in $D_1$. 
\end{lemma}

\begin{proof}
Recall that $\mathrm{f}_\ell(z): = (1-z^\ell)e^{z^\ell}$ for $\ell\in\N$. These functions does not vanishes in $D_1$ so $\log \mathrm{f}_\ell$ is well-defined for the principal branch of $\log$ and 
\[
-\log \mathrm{f}_\ell(z) = \sum_{j\ge 2} \frac{z^{j\ell}}{j} ; \qquad z \in D_1. 
\]
Then, since $\E \Lambda_\ell = d^\ell/\ell$ for $\ell \in \N$, we compute
\[\begin{aligned}
 \sum_{k\in\N} \frac{z^{k}}{k d^{k/2}}  \sum_{\ell | k , \ell<k} \ell \, \E \Lambda_\ell
&  = \sum_{\ell \in\N} \frac{d^\ell}{\ell} \sum_{j\ge 2} \frac{z^{j\ell}}{j d^{j\ell/2}} \\
& = - \sum_{\ell \in\N} \frac{d^\ell}{\ell}  \log \mathrm{f}_\ell(z/\sqrt d) 
\end{aligned}\]
where all sums converge absolutely in $D_1$; e.g.~using the bound 
$\big| \log \mathrm{f}_\ell(z)  \big| \le \frac{r^{2\ell}}{1-r^\ell}$ valid for $z\in D_r$. 
Taking $\exp$ and using formula \eqref{expmd}, this proves the claim. 
\end{proof}

The characteristic polynomial of  a $n\times n$ random matrix $A$ satisfies for $z\in\C$, 
\[
\chi_n(z) = \det(1-zA) =\sum_{k\le n} z^k  \Delta_k^{(n)} 
\]
where  {$\Delta_0^{(n)} =1$} and for $k\ge 1$,
\begin{equation} \label{Deltatr}
  {\Delta_k^{(n)} }=p_k\big(-\tr (A), \cdots, -\tr (A^k)\big)
\end{equation}
 {where $p_k$ is a (multivariate) polynomial of degree $k$ independent of $n$.} 
These polynomials arise by (formally) identifying the power series
\begin{equation}\label{cumulants}
\exp\bigg(\sum_{k\in\N} \frac{z^k}{k} x_k  \bigg) = 1 + \sum_{k\in\N} p_k(x_1,\cdots,x_k) z^k , \qquad (x_1,x_2,\cdots) \in \C^\infty. 
\end{equation}
In particular, they have the following property for every $k\in\N$;
\[
p_k\big(z x_1, \cdots , z^k x_k\big) = z^k  p_k\big(x_1, \cdots ,x_k\big)
\qquad\text{for $z\in\C$}
\]
and 
\[
\big| p_k(x_1,\cdots, x_k) \big| \le p_k(|x_1|,\cdots, |x_k|) .
\]
Hence, if $\sum_{k\in\N} \frac{|x_k|}{k} r^k<\infty$ for $r>0$, then both sums in \eqref{cumulants} are absolutely convergent for $z\in\overline{D_r}$.

The underlying idea is that for $d\in\N$ fixed, by formula \eqref{Deltatr}, Theorem \ref{thm:trace} (1) and the continuous mapping theorem, for every $k\in \N$
\[
\big(\Delta_1^{(n)}, \dotsc, \Delta_k^{(n)}\big) \xrightarrow[n \to \infty]{\mathrm{law}} 
\big(P_1, \dotsc, P_k \big) 
\]
where
\[
P_k = p_k\bigg(-\Lambda_1, \dotsc, -\sum_{\ell | k} \ell \Lambda_\ell \bigg) . 
\]
Moreover, by Proposition~\ref{prop:tightness}, the characteristic polynomial $\{\chi_n\}$ is tight in $D_{1/\sqrt{d}}$. Thus, by Lemma~\ref{lem:finite_convergence},
\[
\chi_n(z)\xrightarrow[n \to \infty]{\mathrm{law}} 1 + \sum_{k\in\N} P_k z^k
\]
locally uniformly for $z\in D_{1/\sqrt{d}}$. 
However, we cannot directly identify the limit from this arguments since  the series $\sum_{k\in\N} \frac{z^k}{k} \sum_{\ell|k} \ell \Lambda_\ell  $ does not converge for $z\in D_{1/\sqrt{d}}$ (cf.~Remark~\ref{rk:d1}). Hence, we give a modified version of this argument which also applies in the regime where $d(n) \to\infty$.

\begin{proof}[Proof of Theorem \ref{thm:mc}]
Instead of the characteristic polynomial, we consider the function
\[
f_n(z) = \frac{\widehat\chi_n(z)}{z-1/\sqrt d} \E e^{-Y_d(z)}  , \qquad z\in D_1, n\in\N .
\]
By Lemma~\ref{lem:mean} and since the rescaled characteristic polynomial has a trivial root at $z= 1/\sqrt d$, these are still well-defined random analytic functions on $D_1$.
Moreover, by  {Lemma~\ref{lemma:technical}}, $\E \|\widehat\chi_n\|_{L^\infty(D_{e^{-t}})} \le C_t$ for $t>0$
so that by Cauchy's formula, 
\[
\E \| f_n\|_{L^\infty(D_{e^{-2t}})} \le  \int_{0}^{2\pi}  \frac{\E\big|\widehat\chi_n(e^{-t+i\theta}) \E e^{-Y_d(e^{-t+i\theta})}\big|}{|e^{-t}-1/\sqrt d| |e^{-t}-e^{-2t}|} \frac{d \theta}{2\pi} 
\]
which is bounded uniformly for $n\in\N$. This shows that $\{f_n\}$ is tight. 
We can expand
\[ 
f_n(z)= 1+ \sum_{k\in\N} a_k^{(n)} z^k ; \qquad 
a_k^{(n)} = p_k\bigg(-\frac{\tr (A) -\E\Lambda_1}{\sqrt d} , \dots, -\frac{\tr (A^k)-\sum_{\ell | k} \ell \E\Lambda_\ell}{\sqrt{d^k}}\bigg) . 
\]
 {Note that since $f_n$ is analytic, the series for $f_n$ is convergent on $D_1$.} This follows from \eqref{cumulants} and the facts that 
\[
 \widehat\chi_n(z) = \exp\bigg(- \sum_{k\ge 1} \frac{z^k}{kd^{k/2}} \tr(A^k)  \bigg)/\sqrt d
\]
(the sum converges at least for $z$ in a small neighborhood of $0$ since we have the crude bound $|\tr(A^k) | \le (nd)^k$ because the entries of $A$ are bounded by $d$)
and according to Lemma~\ref{lem:mean}, 
\[
\frac{\E e^{-Y_d(z)}}{1-z \sqrt d} =  \exp\bigg( \sum_{k\in\N} \frac{z^{k}}{k d^{k/2}}  \sum_{\ell | k } \ell \, \E \Lambda_\ell \bigg), \qquad z\in D_{1/\sqrt d} . 
\]

Hence, for a fixed $d\in\N$, by Theorem \ref{thm:trace} (1) and the continuous mapping theorem, it holds for every $k\in \N$
\[
 \big(a_1^{(n)}, \dotsc, a_k^{(n)}\big) \xrightarrow[n \to \infty]{\mathrm{law}}\big(a_1, \dotsc, a_k \big) ,
\]
where 
\[
a_k = d^{-k/2} p_k\bigg(-\overline{\Lambda_1}, \dots , - \sum_{\ell | k} \ell    \overline{\Lambda_\ell} \bigg) .
\]

By Lemma~\ref{lem:finite_convergence}, we conclude that 
\[
f_n(z) = \frac{\widehat\chi_n(z)}{z-1/\sqrt d} \E e^{-Y_d(z)}
 \xrightarrow[n \to \infty]{\mathrm{law}}  f(z) = 1+ \sum_{k\in\N} a_k z^k.
\]
Moreover, from the proof of Proposition~\ref{prop:PAF}, it holds almost surely 
$ \sum_{k\in\N} \frac{r^k}{k d^{k/2}}  \big| \sum_{\ell | k} \ell \overline{\Lambda_\ell} \big| <\infty$ for any $r<1$. By \eqref{cumulants} and the subsequent observation, it follows that 
\[
e^{- Y_d(z)}=  1+ \sum_{k\in\N} a_k z^k
\]
is the weak-limit of the sequence $\{f_n\}$. This completes the proof.
\end{proof}

\begin{proof}[Proof of Theorem \ref{thm:gmc}]
This is a variant of the previous argument. We consider the function
\[
f_n(z) = \frac{\widehat\chi_n(z)}{z-1/\sqrt d}  , \qquad z\in D_1, n\in\N .
\]

By Proposition~\ref{prop:tightness} and Cauchy's formula, $\{f_n\}$ is tight (the fact that $1/\sqrt{d} \to 0$ is not relevant).
In this case, we have the expansion
\[ 
f_n(z)= 1+ \sum_{k\in\N} a_k^{(n)} z^k ; \qquad 
a_k^{(n)} = p_k\bigg(-\frac{\tr (A) -d}{\sqrt d} , \dots, -\frac{\tr (A^k)-d^k}{\sqrt{d^k}}\bigg) . 
\]
Hence, by Theorem \ref{thm:trace} (2) and the continuous mapping theorem, in the regime where $d=d(n)\to\infty$ and $d(n) = n^{o(1)}$, it holds for every $k\in \N$
\[
 \big(a_1^{(n)}, \dotsc, a_k^{(n)}\big) 
 \xrightarrow[n \to \infty]{\mathrm{law}}
 \big(a_1, \dotsc, a_k \big) ,
 \]
 where
 \[
  a_k=p_k\left(N_1, \sqrt{2} N_2-1,\dotsc,\sqrt{k} N_k -\mathbf{1}_{\{k \text{ is even}\}}\right) .
\]
using the symmetry of $N_k$. Now, we verify that for $z\in D_1$,
\[
\sqrt{1-z^2} e^{X_\infty(z)} = \exp\bigg( \sum_{k\in\N} \frac{z^k}{k} \big(\sqrt{k} N_k -\mathbf{1}_{\{k \text{ is even}\}}\big) \bigg),
\]
where the sum converges absolutely for $z\in\overline{D_r}$ for any $r<1$. Hence, we conclude that if $d(n) \to \infty$ and $d(n) = n^{o(1)}$, then
\[
f_n(z) = \frac{\widehat\chi_n(z)}{z-1/\sqrt d}  \xrightarrow[n \to \infty]{\mathrm{law}} \sqrt{1-z^2} e^{X_\infty(z)}.
\]
This proves the claim.
\end{proof}

\section{Proof of spectral gaps}
 {In the statement of Theorem \ref{cor:spectral_gap_fixed_d}, we recalled the main models for random regular digraphs: (i) is the sum-of-permutation model, (ii) is the sum-of-permutation model conditioned on being simple and (iii) is the uniform directed regular graph.} 
It was proved in \cite{janson1995random} and \cite{molloy19971} that the two  models (ii) and (iii) are \textit{contiguous}, which means if a sequence of events happens asymptotically almost surely in model,  it holds for the other model as well.
For fixed $d$, our  Theorem \ref{thm:mc} can be applied to show a spectral gap result for the  {three} models.

 {The proof is slightly different than \cite[Theorem 1.1]{bordenave2020convergence} and \cite[Theorem 2.6]{coste2021sparse} since we directly rely on the Hurwitz theorem as in \cite {shirai2012limit}.} We identify multisets with integer-valued Radon measures  and endow the space of multisets with the topology of
vague convergence, and the space of random multisets with the topology of weak convergence with respect to the topology of vague convergence.

\begin{prop}[Proposition 2.3 in \cite{shirai2012limit}]\label{prop:random_zero}
	Let $f_n$ be a sequence of random functions in $D_r$ converging in law to a  function $f$ that is  {almost surely not identically zero}. Let $\Phi_n$ and $\Phi$ be the random multisets of the zeros of $f_n, f$ in $D_r$, respectively. Then $\Phi_n$ converges in law towards $\Phi$.
\end{prop}

\begin{proof}[Proof of Theorem \ref{cor:spectral_gap_fixed_d} ]
	We first prove the statement for case (i): the sum of  random permutations.   {Let $f(z)$ be defined by the formula on the right hand side of Theorem \ref{thm:mc}, which has a simple zero at $d^{-1/2}$}. By Proposition \ref{prop:random_zero}, with  {probability tending to 1}, $\widehat\chi_n(z)$ has only one  zero at $1/\sqrt{d}$, which implies  $A/\sqrt{d}$ has no eigenvalue outside $D_{1/r}$ except $\sqrt{d}$.  For any $\varepsilon>0$, take $\frac{1}{1+\varepsilon/\sqrt{d}}<r<1$, then as $n\to\infty$,
\begin{align*}
\mathbf{P}(|\lambda_2| > \sqrt{d} + \varepsilon)\leq \mathbf{P}\left(\frac{|\lambda_2|}{\sqrt{d}} > \frac{1}{r}\right)\to 0.
\end{align*}

It was  shown in \cite{molloy19971} that the probability that  {the digraph associated to $A$ is simple} is bounded from below when $d$ is fixed. Therefore $|\lambda_2| > \sqrt{d}$ holds asymptotically almost surely for case (ii). By contiguity between model (ii) and  model (iii), it holds for uniform random $d$ regular digraphs as well.
\end{proof}


When $d(n)\to\infty$, there is no contiguity result in the literature between the two models (ii) and (iii), and we are not aware of any result to provide a constant lower bound on the probability that $A=\sum_{i=1}^{d(n)} P^{(i)}$ being simple. So our Theorem \ref{thm:gmc} only gives a result for sums of random permutations.

\begin{proof}[Proof of Theorem \ref{cor:spectral_gap_growing_d}]
In Theorem \ref{thm:gmc}, $\frac{\widehat\chi_n(z)}{z-1/\sqrt{d}}$ converges in law to the limiting function $ \sqrt{1-z^2} e^{X_\infty(z)}$, which has no root inside $D_r$ for any $r\in (0,1)$. From Proposition \ref{prop:random_zero}, with high probability $\frac{\widehat\chi_n(z)}{z-1/\sqrt{d}}$ has  no  roots in $D_{r}$ for any $0<r<1$. Then   $A/\sqrt{d(n)}$ has no eigenvalue outside $D_{1/r}$ except for $\sqrt{d(n)}$. Therefore \eqref{eq:spectal_gap_limit} holds.
\end{proof}

\appendix\section{Appendix: a Poisson CLT} \label{sec:Poi}

Let $\Lambda\sim \mathrm{Poisson}(\lambda)$. 
Recall that the Laplace transform of $\Lambda$ is
\begin{equation*} 
\E \exp(z(\Lambda-\lambda)) = \exp\big(\lambda (e^z-1-z)\big) , \qquad z\in\C. 
\end{equation*}
Let $\mu_k(\lambda) = \E (\Lambda-\lambda)^k $ be the central moments of $\Lambda$ for $k\in\N$. We have for $k\in \N$  \[
\mu_k(\lambda) /k! = [z^k]  \exp\big(\lambda (e^z-1-z)\big) .
\]
In particular, $\mu_1=0$ and for $k\in\N$,
\begin{align} \label{eq:Poisson_moments}
\begin{cases}
\mu_{2k}(\lambda) =m_k \lambda^k + p_k(\lambda) , &
 p_k \in \mathcal P_{k-1} \\
\mu_{2k+1} \in  \mathcal P_{k}
\end{cases}
\end{align}
where $m_k =  \tfrac{(2k)!}{2^kk!}$, and $\mathcal P_{k} = \{\text{polynomials of degree }\le k\}$.
From this fact, we deduce that for every $k\in\N$,
\begin{equation} \label{PoiCLT}
\frac{\mu_k(\lambda)}{\lambda^{k/2}}= \E \bigg(\frac{\Lambda-\lambda}{\sqrt \lambda}\bigg)^k  
 \xrightarrow[\lambda \to \infty]{} 
  \E N^k =
 \begin{cases} m_{k/2}  & k\text{ even} \\
 0 & k\text{ odd}
 \end{cases} ,
\end{equation}
where $N$ is a standard Gaussian.
This is a CLT in the sense of moments for a Poisson random variable.

The goal of this appendix is to extend this results to a collection of asymptotically Poisson random variables with large parameters.
First recall that we can expand for $k\in\N$,
\begin{equation} \label{momexp}
(x-\lambda)^k =\textstyle{\sum_{j=0}^k}(x)_{k-j} q_{k,j}(\lambda) ; \qquad q_{k,j} \in \mathcal P_{j} .
\end{equation}
Then, according to Lemma~\ref{lem:Poi}, we have 
\begin{equation} \label{momPoi}
\mu_k(\lambda) =  \sum_{j\le k} q_{k,j}(\lambda) \E(\Lambda)_{k-j}
=  \sum_{j\le k}q_{k,j}(\lambda)  \lambda^{k-j} .
\end{equation}
Remarkably, from \eqref{eq:Poisson_moments}, the RHS is a polynomial in $\lambda$ of degree $\le k/2$.

\begin{lemma} \label{lem:PoiCLT}
Fix $\ell\in\N$ and for $j\in [\ell]$, let $(\lambda_{j,n})_{n\in\N}$  be a sequence in $\mathbb R_+$ such that $\lambda_{j,n} \to \infty$ as $n\to\infty$.
Let $(Q_{j,n})_{j\in[\ell], n\in\N}$ be a sequence of random vectors such that its joint factorial moments satisfy for  any $\mathbf k\in\N^\ell$, 
\begin{equation} \label{Poimom}
\E\big[ (Q_{1,n})_{k_1}\cdots  (Q_{\ell,n})_{k_\ell} \big]=  \lambda_{1,n}^{k_1}\cdots \lambda_{\ell,n}^{k_\ell} (1+ \epsilon_{\mathbf k,n})
\end{equation}
where $ \epsilon_{\mathbf k,n} \in\mathbb R$ and for $\mathbf k\in\N^\ell$, 
\begin{equation} \label{PoiCLTcond}
\lambda_{1,n}^{k_1/2} \cdots \lambda_{\ell,n}^{k_\ell/2}  \cdot \max_{\mathbf j\le \mathbf  k}|\epsilon_{\mathbf j,n}| \to 0. 
\end{equation}
Then
\[
\left(\frac{Q_{1,n}-\lambda_{1,n}}{\sqrt{\lambda_{1,n}}}, \dots, \frac{Q_{\ell,n}-\lambda_{\ell,n}}{\sqrt{\lambda_{\ell,n}}} \right)  \xrightarrow[n \to \infty]{\mathrm{law}} \big( N_1,\dots, N_\ell \big)
\]
and in the sense of moments, where $(N_i)_{i\in[\ell]}$ are i.i.d. standard Gaussians. 
\end{lemma}

\begin{proof}
Let us denote $\mu_{\mathbf k,n} = \E\big[ (Q_{1,n}-\lambda_{1,n})^{k_1} \cdots (Q_{\ell,n}-\lambda_{\ell,n})^{k_\ell}\big]  $ for $\mathbf k\in\N^\ell$. 
Using \eqref{momexp} and \eqref{Poimom}, we can write 
\[\begin{aligned}
\mu_{\mathbf k,n} 
&= \sum_{\mathbf j\le \mathbf k} \prod_{i\le\ell} q_{k_i,j_i}(\lambda_{i,n})\E\big[ (Q_{1,n})_{k_1-j_1}\cdots (Q_{\ell,n})_{k_\ell-j_\ell}\big] \\
&=  \sum_{\mathbf j\le \mathbf k} \prod_{i\le\ell} q_{k_i,j_i}(\lambda_{i,n})\lambda_{i,n}^{k_i-j_i} (1+ \epsilon_{\mathbf{k-j},n}) .
\end{aligned}\]
Then, the conditions \eqref{PoiCLTcond} imply that for any fixed $\mathbf k\in\N^\ell$,
\[
\mu_{\mathbf k,n} 
=  \prod_{i\le \ell} \bigg( \sum_{j\le k_i} q_{k_i,j}(\lambda_{i,n})\lambda_{i,n}^{k_i-j} \bigg)
+ \underset{n\to\infty}{o_{\mathbf k}}\big(\lambda_{1,n}^{k_1/2} \cdots \lambda_{\ell,n}^{k_\ell/2} \big) .
\]
Here we used that $q_{k,j} \in \mathcal P_{j}$.
Hence using the combinatorial identity \eqref{momPoi},
\[
\mu_{k,n}  = \mu_{k_1}(\lambda_{1,n})\cdots \mu_{k_\ell}(\lambda_{\ell,n}) 
+ \underset{n\to\infty}{o_{\mathbf k}}\big(\lambda_{1,n}^{k_1/2} \cdots \lambda_{\ell,n}^{k_\ell/2} \big) .
\]
From \eqref{PoiCLT}, we conclude that for every $\mathbf k\in\N^\ell$,
\[
\frac{\E\big[ (Q_{1,n}-\lambda_{1,n})^{k_1} \cdots (Q_{\ell,n}-\lambda_{\ell,n})^{k_\ell}\big]}{\lambda_{1,n}^{k_1/2}\cdots\lambda_{\ell,n}^{k_\ell/2}}  \xrightarrow[n \to \infty]{} \E\big[ N_1^{k_1}\cdots N_\ell^{k_\ell} \big] .
\]
This completes the proof, since the law of $(N_i)_{i\in[\ell]}$  is characterized by its joint moments.
\end{proof}

\section{Appendix: sums of non-uniform random permutations}\label{app:ewens}

In this paper, we studied random digraphs arising as sums of $n\times n$ uniform permutations. A crucial technical argument in our analysis was that the rows of the matrix $A = P^{(1)} + \dotsb + P^{(d)}$ are exchangeable; but this property is in fact very strong. It would not hold if the permutation matrices $P^{(i)}$ were not uniformly distributed on $S_n$. This is typically what happens when they are skewed towards having more or fewer short cycles. For example, the Ewens distribution with parameter $\theta$ is defined as follows; we say that a random permutation $\pi$ is $\mathrm{Ewens}(\theta)$-distributed if 
\begin{equation*}
  \forall \sigma \in S_n, \qquad \mathbb{P} (\pi = \sigma) = \frac{\theta^{\mathrm{cyc}(\sigma)}}{\theta(\theta+1)\dotsb(\theta + n-1)},
\end{equation*}
where $\mathrm{cyc}(\sigma)$ is the number of cycles in the cycle decomposition of $\sigma$. When $\theta=1$, the random permutation $\pi$ is uniform; when $\theta>1$ (respectively, $<1$), it is skewed towards having more cycles (respectively, less), resulting in a directed graph with a different local structure. Crucially, Ewens-distributed random permutations with $\theta\neq 1$ are invariant by \emph{conjugation} by any permutation matrix, but not invariant by multiplication by a permutation matrix. 

We were able to identify the asymptotics of the traces of $A = P^{(1)} + \dotsb + P^{(d)}$ is this model, with $\theta$ fixed and not depending on $n$; indeed, Theorem \ref{thm:trace} still holds, with the limiting random variables $\Lambda_\ell$ being replaced with the following ones: 
$$ \Theta_\ell \sim \mathrm{Poisson}\left(\frac{d^\ell + d(\theta-1)}{\ell}\right).  $$
A proof is available on demand; however, the proof of the tightness of the sequence $\det(I - zA_n)$ is still under exploration. We mention this fact, since the asymptotic behaviour of the spectrum of $A$ in the Ewens case displays some unusual features; when $\theta$ is sufficiently large (this might depend on $n$), the whole limiting shape of the eigenvalues seems not rotation-invariant, with a previously unseen pattern of eigenvector localization --- see Figure \ref{fig:ewens}. For comparison, Figure \ref{fig:gini} displays the eigenvalues of our model of sums-of-permutation matrices, together with an example of a real Ginibre spectrum.

We note that the continuous counterpart of the Ewens measure on permutations, is a generalization of the Haar measure on the unitary group, which is called the Hua-Pickrell measure \cite{hua1963harmonic,pickrell1987measures}, also known as the  circular Jacobi ensembles \cite{bourgade2011ewens}. It is  also an interesting question to study the limiting spectral distribution for sums of independent unitary matrices sampled from the Hua-Pickrell distribution beyond the Haar case  \cite{basak2013limiting}. 

\begin{figure}
\centering
\begin{tabular}{ccc}
  \includegraphics[width=0.3\textwidth]{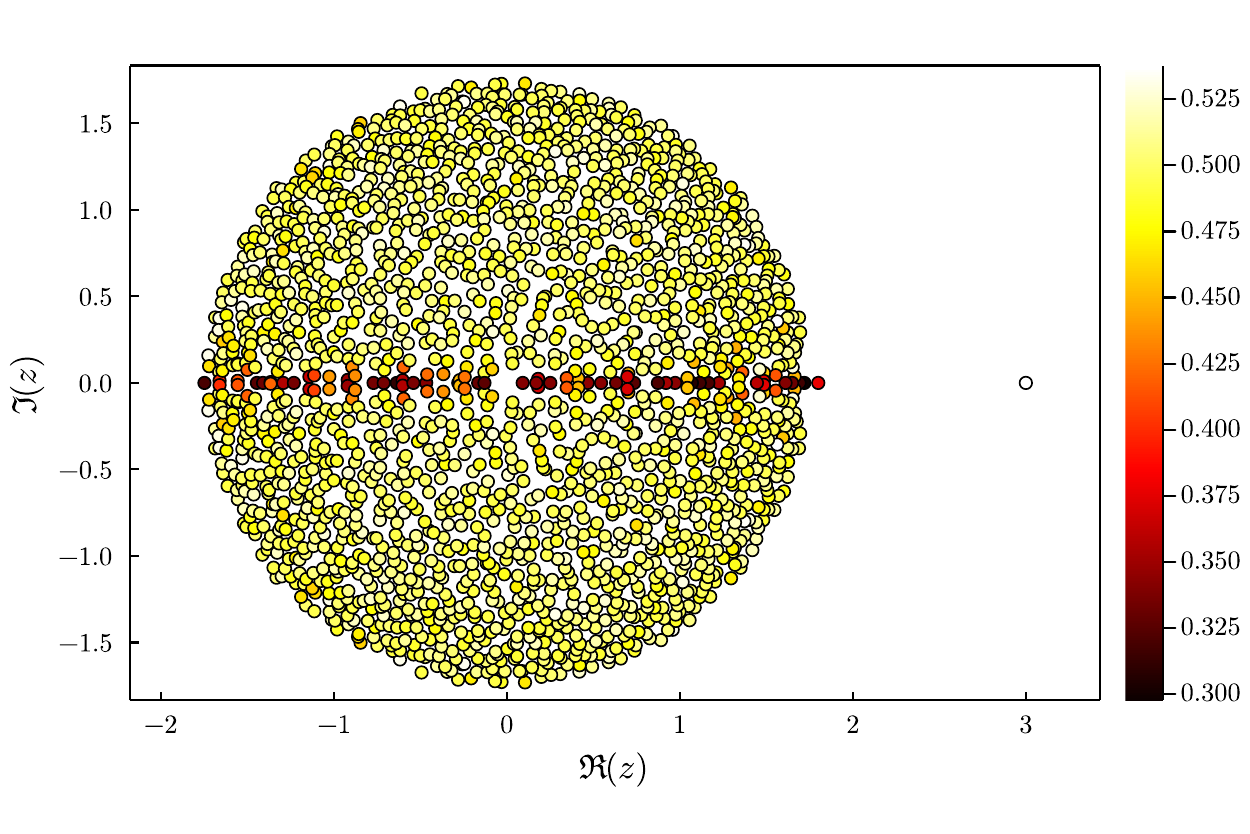}&\includegraphics[width=0.3\textwidth]{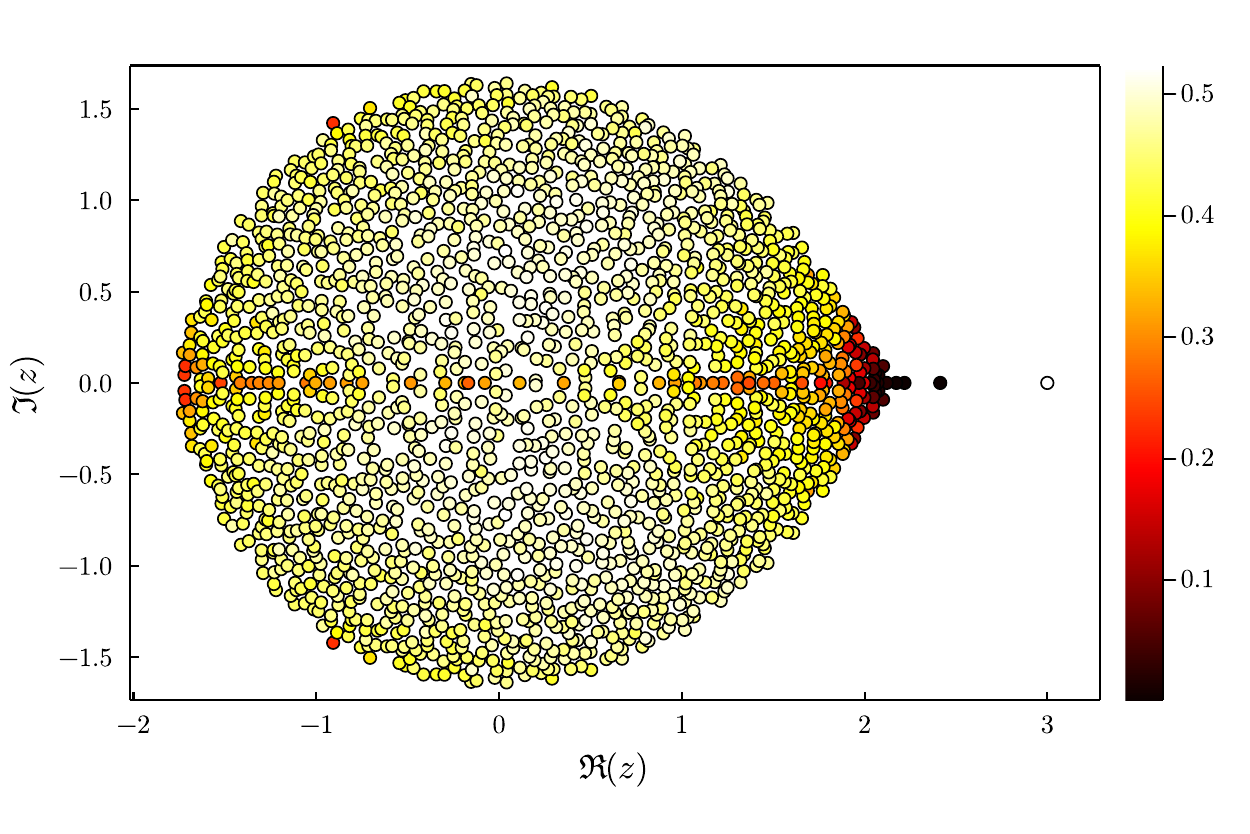} &\includegraphics[width=0.3\textwidth]{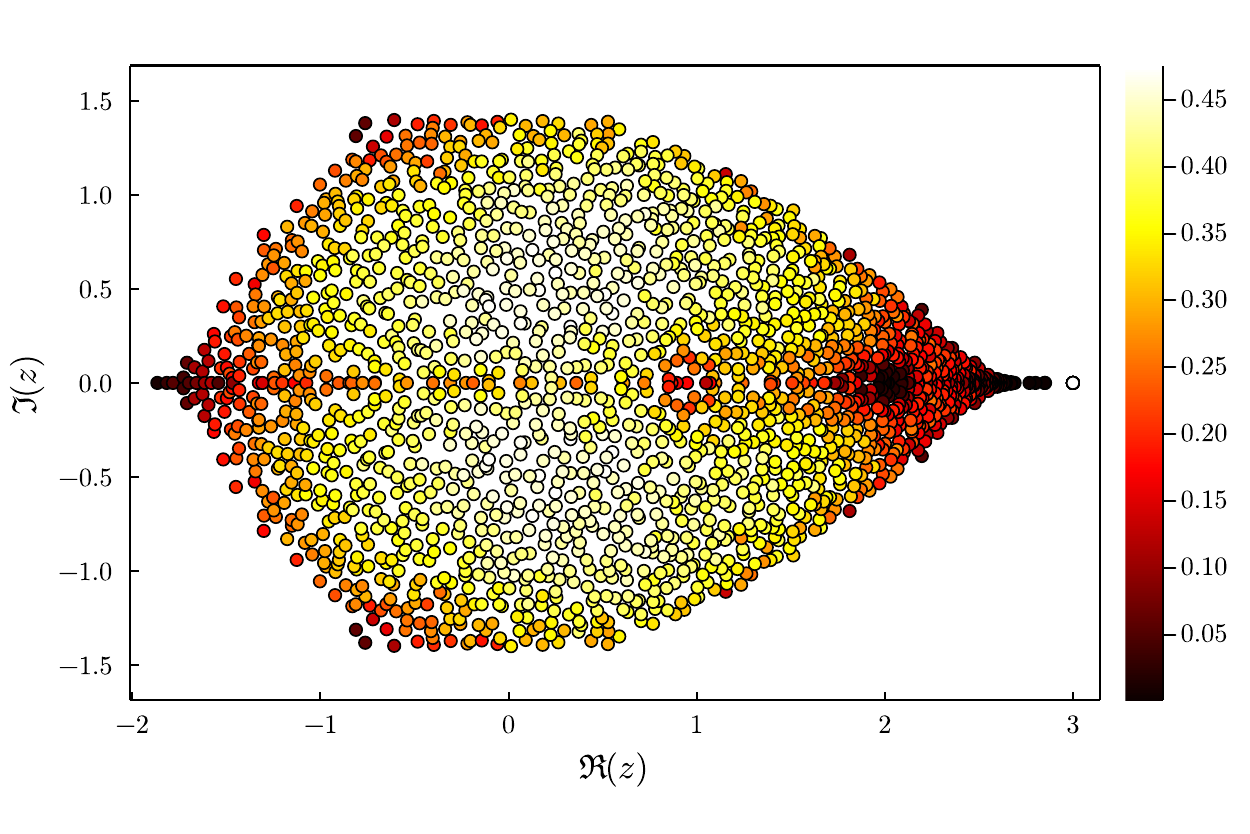} \\
  $\theta = 0.0001$ & $\theta = 100$ & $\theta = 500$ \\
\end{tabular}
\caption{These figures display the $n=2000$ complex eigenvalues of sums of $d=3$ permutations matrices, the underlying permutations being $\mathrm{Ewens}(\theta)$-distributed on $\mathfrak{S}_{2000}$ for various $\theta$. Each eigenvalue is coloured according to the degree of localization of its corresponding right-eigenvector; here, we measure the localisation of a vector $\varphi \in \mathbb{C}^n$ using the Inverse Participation Ratio $\mathrm{IPR}(\varphi) = |\varphi|_2^4 / n|\varphi|_4^4 \leqslant 1$. An IPR equal to 1 means that $\varphi$ is constant up to phases (pure delocalization); an IPR equal to $1/n$ means that  $\varphi$ is a multiple of a Dirac (pure localization). Here, we see that the eigenvectors with eigenvalues close to the real axis are more localized, a phenomenon already visible for classical matrix ensembles such as Ginibre (see Figure \ref{fig:gini}). }\label{fig:ewens}
\end{figure}

\begin{figure}
  \centering
  \begin{tabular}{ccc}
    \includegraphics[width=0.3\textwidth]{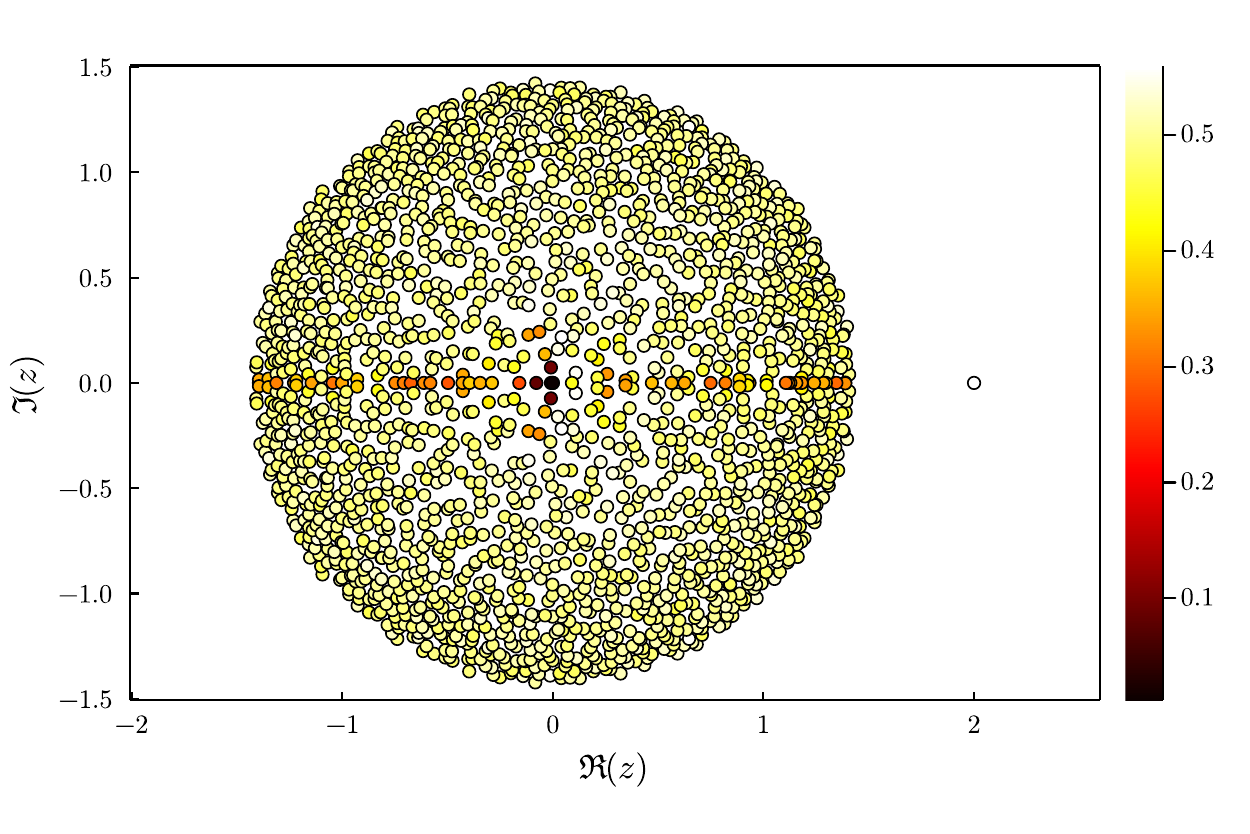}&\includegraphics[width=0.3\textwidth]{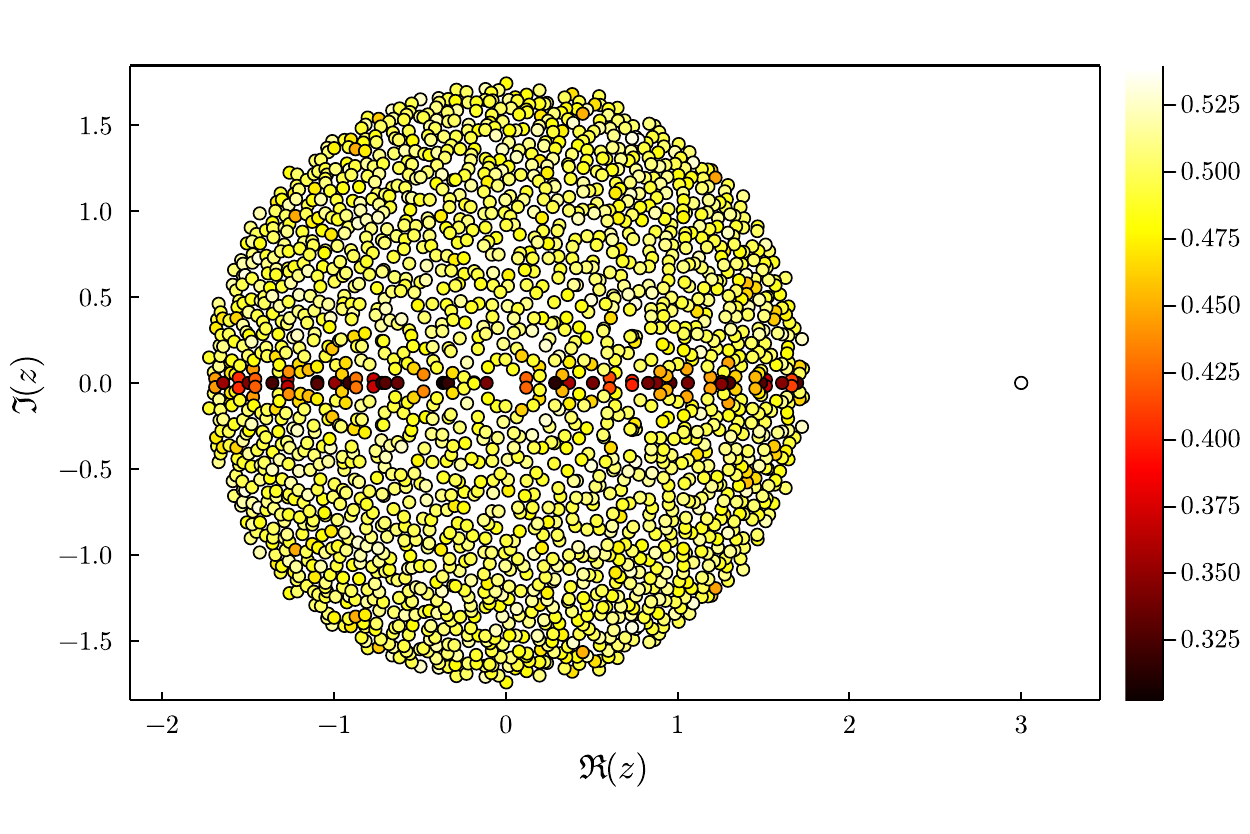} &\includegraphics[width=0.3\textwidth]{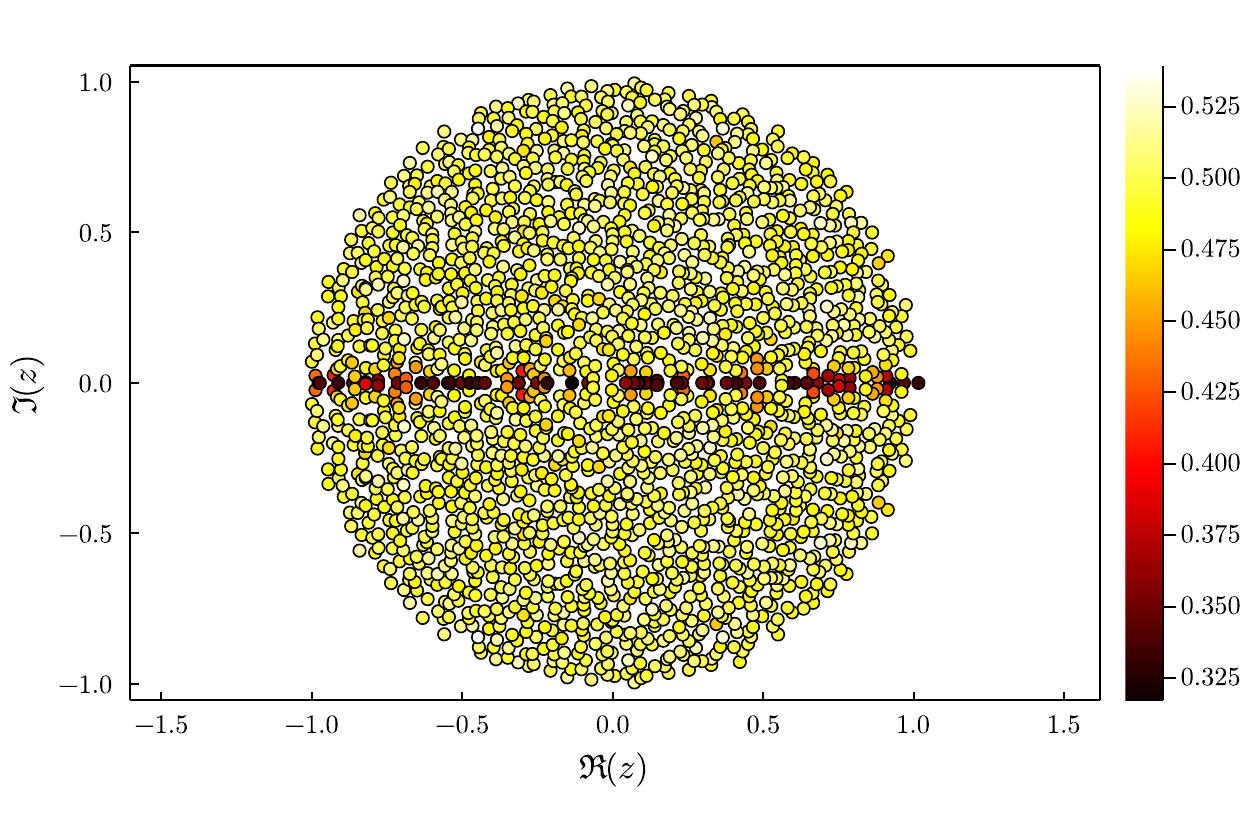} \\
    $d=2$ & $d = 3$ & Ginibre ensemble \\
  \end{tabular}
  \caption{These figures display the $n=2000$ complex eigenvalues of sums of $d$ uniform permutations matrices, with eigenvector localization depicted as the Figure \ref{fig:ewens}. We plotted in the third panel the eigenvalues of a Real Ginibre matrix (the entries are iid $N_{\mathbb{R}}(0, 1/n)$) for comparison.}\label{fig:gini}
  \end{figure}

\bibliographystyle{plain}

\begin{thebibliography}{10}

\bibitem{ABB17}
Louis-Pierre Arguin, David Belius, and Paul Bourgade.
\newblock Maximum of the characteristic polynomial of random unitary matrices.
\newblock {\em Communications in Mathematical Physics}, 349(2):703--751, 2017.

\bibitem{arratia1992cycle}
Richard Arratia and Simon Tavar{\'e}.
\newblock The cycle structure of random permutations.
\newblock {\em The Annals of Probability}, pages 1567--1591, 1992.

\bibitem{bahier2019characteristic}
Valentin Bahier.
\newblock Characteristic polynomials of modified permutation matrices at
  microscopic scale.
\newblock {\em Stochastic Processes and their Applications},
  129(11):4335--4365, 2019.

\bibitem{bahier2019number}
Valentin Bahier.
\newblock On the number of eigenvalues of modified permutation matrices in
  mesoscopic intervals.
\newblock {\em Journal of Theoretical Probability}, 32(2):974--1022, 2019.

\bibitem{bahier2021smooth}
Valentin Bahier and Joseph Najnudel.
\newblock On smooth mesoscopic linear statistics of the eigenvalues of random
  permutation matrices.
\newblock {\em Journal of Theoretical Probability}, 35(3):1640--1661, 2022.

\bibitem{basak2018circular}
Anirban Basak, Nicholas Cook, and Ofer Zeitouni.
\newblock Circular law for the sum of random permutation matrices.
\newblock {\em Electronic Journal of Probability}, 23:1--51, 2018.

\bibitem{basak2013limiting}
Anirban Basak and Amir Dembo.
\newblock Limiting spectral distribution of sum of unitary and orthogonal
  matrices.
\newblock {\em Electronic Communications in Probability}, 18:1--19, 2013.

\bibitem{basak2019circular}
Anirban Basak and Mark Rudelson.
\newblock The circular law for sparse non-{H}ermitian matrices.
\newblock {\em The Annals of Probability}, 47(4):2359--2416, 2019.

\bibitem{ben2015fluctuations}
G{\'e}rard Ben~Arous and Kim Dang.
\newblock On fluctuations of eigenvalues of random permutation matrices.
\newblock {\em Annales de l'IHP Probabilit{\'e}s et statistiques},
  51(2):620--647, 2015.

\bibitem{bordenave2020new}
Charles Bordenave.
\newblock A new proof of {F}riedman's second eigenvalue theorem and its
  extension to random lifts.
\newblock {\em Annales Scientifiques de l'{\'E}cole Normale Sup{\'e}rieure},
  4(6):1393--1439, 2020.

\bibitem{bordenave2012around}
Charles Bordenave and Djalil Chafa{\"\i}.
\newblock Around the circular law.
\newblock {\em Probability surveys}, 9:1--89, 2012.

\bibitem{bordenave2020convergence}
Charles Bordenave, Djalil Chafa{\"\i}, and David Garc{\'\i}a-Zelada.
\newblock Convergence of the spectral radius of a random matrix through its
  characteristic polynomial.
\newblock {\em Probability Theory and Related Fields}, 182:1163--1181, 2022.

\bibitem{bourgade2011ewens}
Paul Bourgade, Ashkan Nikeghbali, and Alain Rouault.
\newblock Ewens measures on compact groups and hypergeometric kernels.
\newblock In {\em S{\'e}minaire de probabilit{\'e}s XLIII}, pages 351--377.
  Springer, 2011.

\bibitem{CMN18}
Reda Chhaibi, Thomas Madaule, and Joseph Najnudel.
\newblock On the maximum of the {C$\beta$E} field.
\newblock {\em Duke Mathematical Journal}, 167(12):2243--2345, 2018.

\bibitem{cook2019circular}
Nicholas Cook.
\newblock The circular law for random regular digraphs.
\newblock {\em Annales de l'Institut Henri Poincar{\'e}, Probabilit{\'e}s et
  Statistiques}, 55(4):2111--2167, 2019.

\bibitem{cook2018size}
Nicholas Cook, Larry Goldstein, and Tobias Johnson.
\newblock Size biased couplings and the spectral gap for random regular graphs.
\newblock {\em The Annals of Probability}, 46(1):72--125, 2018.

\bibitem{cook2020maximum}
Nicholas Cook and Ofer Zeitouni.
\newblock Maximum of the characteristic polynomial for a random permutation
  matrix.
\newblock {\em Communications on Pure and Applied Mathematics},
  73(8):1660--1731, 2020.

\bibitem{cook2017discrepancy}
Nicholas~A Cook.
\newblock Discrepancy properties for random regular digraphs.
\newblock {\em Random Structures \& Algorithms}, 50(1):23--58, 2017.

\bibitem{cook2017singularity}
Nicholas~A Cook.
\newblock On the singularity of adjacency matrices for random regular digraphs.
\newblock {\em Probability Theory and Related Fields}, 167(1):143--200, 2017.

\bibitem{coste2021spectral}
Simon Coste.
\newblock The spectral gap of sparse random digraphs.
\newblock {\em Annales de l'Institut Henri Poincar{\'e}, Probabilit{\'e}s et
  Statistiques}, 57(2):644--684, 2021.

\bibitem{coste2021sparse}
Simon Coste.
\newblock Sparse matrices: convergence of the characteristic polynomial seen
  from infinity.
\newblock {\em Electronic Journal of Probability}, 28:1--40, 2023.

\bibitem{dang2014characteristic}
Kim Dang and Dirk Zeindler.
\newblock The characteristic polynomial of a random permutation matrix at
  different points.
\newblock {\em Stochastic Processes and their Applications}, 124(1):411--439,
  2014.

\bibitem{dumitriu2013functional}
Ioana Dumitriu, Tobias Johnson, Soumik Pal, and Elliot Paquette.
\newblock Functional limit theorems for random regular graphs.
\newblock {\em Probability Theory and Related Fields}, 156(3):921--975, 2013.

\bibitem{dumitriu2020global}
Ioana Dumitriu and Yizhe Zhu.
\newblock Global eigenvalue fluctuations of random biregular bipartite graphs.
\newblock {\em arXiv preprint arXiv:2008.11760}, 2020.

\bibitem{DRSV17}
Bertrand Duplantier, R{\'e}mi Rhodes, Scott Sheffield, and Vincent Vargas.
\newblock Log-correlated {G}aussian fields: an overview.
\newblock {\em Geometry, analysis and probability}, pages 191--216, 2017.

\bibitem{friedman2008proof}
Joel Friedman.
\newblock {\em A Proof of Alon's Second Eigenvalue Conjecture and Related
  Problems}.
\newblock Memoirs of the American Mathematical Society. American Mathematical
  Society, 2008.

\bibitem{friedman1989second}
Joel Friedman, Jeff Kahn, and Endre Szemeredi.
\newblock On the second eigenvalue of random regular graphs.
\newblock In {\em Proceedings of the twenty-first annual ACM symposium on
  Theory of computing}, pages 587--598, 1989.

\bibitem{FG15}
Yan~V Fyodorov and Olivier Giraud.
\newblock High values of disorder-generated multifractals and logarithmically
  correlated processes.
\newblock {\em Chaos, Solitons \& Fractals}, 74:15--26, 2015.

\bibitem{FHK12}
Yan~V Fyodorov, Ghaith~A Hiary, and Jonathan~P Keating.
\newblock Freezing transition, characteristic polynomials of random matrices,
  and the {R}iemann zeta function.
\newblock {\em Physical Review Letters}, 108(17):170601, 2012.

\bibitem{ganguly2020random}
Shirshendu Ganguly and Soumik Pal.
\newblock The random transposition dynamics on random regular graphs and the
  {G}aussian free field.
\newblock {\em Annales de l'Institut Henri Poincar{\'e}, Probabilit{\'e}s et
  Statistiques}, 56(4):2935--2970, 2020.

\bibitem{hough2009zeros}
John~Ben Hough, Manjunath Krishnapur, and Yuval Peres.
\newblock {\em Zeros of Gaussian analytic functions and determinantal point
  processes}, volume~51.
\newblock American Mathematical Soc., 2009.

\bibitem{hua1963harmonic}
Luogeng Hua.
\newblock {\em Harmonic analysis of functions of several complex variables in
  the classical domains}.
\newblock Number~6. American Mathematical Soc., 1963.

\bibitem{huang2021invertibility}
Jiaoyang Huang.
\newblock Invertibility of adjacency matrices for random $d$-regular graphs.
\newblock {\em Duke Mathematical Journal}, 170(18):3977--4032, 2021.

\bibitem{hughes2013random}
Christopher Hughes, Joseph Najnudel, Ashkan Nikeghbali, and Dirk Zeindler.
\newblock Random permutation matrices under the generalized ewens measure.
\newblock {\em The Annals of Applied Probability}, 23(3):987--1024, 2013.

\bibitem{jainsmallest}
Vishesh Jain, Ashwin Sah, and Mehtaab Sawhney.
\newblock The smallest singular value of dense random regular digraphs.
\newblock {\em International Mathematics Research Notices},
  2022(24):19300--19334, 2022.

\bibitem{janson1995random}
Svante Janson.
\newblock Random regular graphs: asymptotic distributions and contiguity.
\newblock {\em Combinatorics, Probability and Computing}, 4(4):369--405, 1995.

\bibitem{johnson2015exchangeable}
Tobias Johnson.
\newblock Exchangeable pairs, switchings, and random regular graphs.
\newblock {\em The Electronic Journal of Combinatorics}, 22(1):P1--33, 2015.

\bibitem{johnson2014cycles}
Tobias Johnson and Soumik Pal.
\newblock Cycles and eigenvalues of sequentially growing random regular graphs.
\newblock {\em The Annals of Probability}, 42(4):1396--1437, 2014.

\bibitem{lambert2020maximum}
Gaultier Lambert.
\newblock Maximum of the characteristic polynomial of the ginibre ensemble.
\newblock {\em Communications in Mathematical Physics}, 378(2):943--985, 2020.

\bibitem{L21}
Gaultier Lambert.
\newblock Mesoscopic central limit theorem for the circular $\beta$-ensembles
  and applications.
\newblock {\em Electronic Journal of Probability}, 26:1--33, 2021.

\bibitem{LP}
Gaultier Lambert and Elliot Paquette.
\newblock Strong approximation of gaussian $\beta$-ensemble characteristic
  polynomials: The hyperbolic regime.
\newblock {\em The Annals of Applied Probability}, 33(1):549--612, 2023.

\bibitem{litvak2017adjacency}
Alexander~E Litvak, Anna Lytova, Konstantin Tikhomirov, Nicole
  Tomczak-Jaegermann, and Pierre Youssef.
\newblock Adjacency matrices of random digraphs: singularity and
  anti-concentration.
\newblock {\em Journal of Mathematical Analysis and Applications},
  445(2):1447--1491, 2017.

\bibitem{litvak2019smallest}
Alexander~E Litvak, Anna Lytova, Konstantin Tikhomirov, Nicole
  Tomczak-Jaegermann, and Pierre Youssef.
\newblock The smallest singular value of a shifted d-regular random square
  matrix.
\newblock {\em Probability Theory and Related Fields}, 173(3):1301--1347, 2019.

\bibitem{litvak2020circular}
Alexander~E Litvak, Anna Lytova, Konstantin Tikhomirov, Nicole
  Tomczak-Jaegermann, and Pierre Youssef.
\newblock Circular law for sparse random regular digraphs.
\newblock {\em Journal of the European Mathematical Society}, 23(2):467--501,
  2020.

\bibitem{meszaros2020distribution}
Andr{\'a}s M{\'e}sz{\'a}ros.
\newblock The distribution of sandpile groups of random regular graphs.
\newblock {\em Transactions of the American Mathematical Society},
  373(9):6529--6594, 2020.

\bibitem{metz2019spectral}
Fernando~Lucas Metz, Izaak Neri, and Tim Rogers.
\newblock Spectral theory of sparse non-{H}ermitian random matrices.
\newblock {\em Journal of Physics A: Mathematical and Theoretical},
  52(43):434003, 2019.

\bibitem{molloy19971}
Michael S.~O. Molloy, Hanna Robalewska, Robert~W. Robinson, and Nicholas~C.
  Wormald.
\newblock 1-factorizations of random regular graphs.
\newblock {\em Random Structures \& Algorithms}, 10(3):305--321, 1997.

\bibitem{najnudel2020secular}
Joseph Najnudel, Elliot Paquette, and Nick Simm.
\newblock Secular coefficients and the holomorphic multiplicative chaos.
\newblock {\em arXiv preprint arXiv:2011.01823}, 2020.

\bibitem{NSW20}
Miika Nikula, Eero Saksman, and Christian Webb.
\newblock Multiplicative chaos and the characteristic polynomial of the {CUE}:
  The ${L}^1$-phase.
\newblock {\em Transactions of the American Mathematical Society},
  373(6):3905--3965, 2020.

\bibitem{PZ17}
Elliot Paquette and Ofer Zeitouni.
\newblock The maximum of the {CUE} field.
\newblock {\em International Mathematics Research Notices},
  2018(16):5028--5119, 2018.

\bibitem{parzanchevski2020ramanujan}
Ori Parzanchevski.
\newblock Ramanujan graphs and digraphs.
\newblock {\em Analysis and Geometry on Graphs and Manifolds}, 461:344, 2020.

\bibitem{pickrell1987measures}
Doug Pickrell.
\newblock Measures on infinite dimensional grassmann manifolds.
\newblock {\em Journal of Functional Analysis}, 70(2):323--356, 1987.

\bibitem{rider2007noise}
Brian Rider and B{\'a}lint Vir{\'a}g.
\newblock The noise in the circular law and the {G}aussian free field.
\newblock {\em International Mathematics Research Notices}, 2007(9):rnm006,
  2007.

\bibitem{rudelson2019sparse}
Mark Rudelson and Konstantin Tikhomirov.
\newblock The sparse circular law under minimal assumptions.
\newblock {\em Geometric and Functional Analysis}, 29(2):561--637, 2019.

\bibitem{shirai2012limit}
Tomoyuki Shirai.
\newblock Limit theorems for random analytic functions and their zeros:
  Dedicated to the late {P}rofessor {Y}asunori {O}kabe (functions in number
  theory and their probabilistic aspects).
\newblock {\em RIMS Kokyuroku Bessatsu}, 34:335--359, 2012.

\bibitem{tikhomirov2019spectral}
Konstantin Tikhomirov and Pierre Youssef.
\newblock The spectral gap of dense random regular graphs.
\newblock {\em The Annals of Probability}, 47(1):362--419, 2019.

\bibitem{wood2012universality}
Philip~Matchett Wood.
\newblock Universality and the circular law for sparse random matrices.
\newblock {\em The Annals of Applied Probability}, 22(3):1266--1300, 2012.

\bibitem{zhu2020second}
Yizhe Zhu.
\newblock On the second eigenvalue of random bipartite biregular graphs.
\newblock {\em Journal of Theoretical Probability}, pages 1--35, 2022.

\end{thebibliography}

\end{document}